\documentclass[a4paper,11pt]{amsart}
\usepackage[pdftex,hyperindex,breaklinks,colorlinks,citecolor=blue]{hyperref}
\usepackage[all]{xy}
\usepackage{amscd}
\usepackage{amsfonts}
\usepackage{a4wide}
\usepackage{amsmath}
\usepackage{amssymb}

\usepackage{fullpage}
\usepackage[applemac]{inputenc} 
\usepackage{mathrsfs}

\newcommand{\C}{\mathbb{C}}

\newcommand{\Pic}{\mathrm{Pic}}

\newcommand{\Div}{\mathrm{Div}}

\newcommand{\sm}{\mathrm{sm}}
\newcommand{\Supp}{\mathrm{Supp}}
\newcommand{\Sing}{\mathrm{Sing}}
\newcommand{\red}{\mathrm{red}}
\newcommand{\nc}{\mathrm{nc}}
\usepackage{tikz-cd}

\newtheorem{thm}{Theorem}[section]
\newtheorem{lema}[thm]{Lemma}
\newtheorem{prop}[thm]{Proposition}
\newtheorem{cor}[thm]{Corollary}

\theoremstyle{definition}

\newtheorem{ej}[thm]{Example}

\newtheorem{defn}[thm]{Definition}

\title{Notas}

\title{The Generalized Torelli Problem through the geometry of the Gauss map}
\author{Sebasti\'an Rahausen Rodr\'iguez}
\address{S. Rahausen \\ Departamento de Matem\'atica, Universidad T\'ecnica Federico Santa Mar\'ia, San Joaqu\'in, Santiago\\Chile}
\email{srahausen@gmail.com}


\begin{document}
	
\begin{abstract}
	Given a non-hyperelliptic curve $C\in\mathscr{M}_g$ and $2\leq n\leq g-2$, we prove that the generic fiber of the Gauss map on $W_n$ has one element and we characterize its multiple locus. Assuming that $C$ doesn't have a $\mathfrak{g}_{n+k+1}^{k+1}$, for $1\leq k\leq n-1\leq g-3$, we solve the problem of reconstructing each $\mathfrak{g}_{n+k}^k$ and the dual hypersurface of the image of its associated morphism, through information encoded in the Gauss map. For this purpose we introduce the notion of $\left(n+k\right)$-intersection loci and we study their dimensions. In the hyperelliptic case we prove that the image of the Gauss map is a union of sets whose closures are birational to their complete $\mathfrak{g}_{n+k}^k$'s, for each $1\leq k\leq n\leq g-1$, and that these also contain a copy of the dual hypersurface of the image of its associated morphism. From the case $k=n$ we deduce that the closure of the image of the Gauss map is birational to $\mathbb{P}^n$.
\end{abstract}
\maketitle
\section{Introduction}

Torelli's Theorem, introduced by the Italian mathematician Ruggiero Torelli (1884-1915), is a classic result in algebraic geometry that states that a smooth complex projective curve $C$ (compact Riemann surface) is completely determined by its Jacobian variety $J\left(C\right)$ and its theta divisor $\Theta\simeq W_{g-1}$. More precisely, given two smooth projective curves $C,C'\in\mathscr{M}_g$, then $\left(J\left(C\right),\Theta\right)$ and $\left(J\left (C'\right),\Theta'\right)$ are isomorphic as principally polarized abelian varieties if and only if $C$ and $C'$ are isomorphic as algebraic curves \cite{T}. The same statement is true over any algebraically closed field of characteristic $p\neq2$ \cite{A}. This theorem has had several important generalizations such as \cite{BR, CG, C, D, DO, RH, L, PZ}. In particular the \textit{``Generalized Torelli Problem''} that we will study in this paper. To explain what this problem consists of, we will introduce some definitions. 

Let $C\in\mathscr{M}_g$ be a smooth projective curve over an algebraically closed field $\mathbb{F}$ of characteristic zero. For $n\geq1$ we denote $C^{(n)}$ the \textit{symmetric product of order} $n$, which is the quotient of the cartesian product $C^n$ by the natural action of the symmetric group of degree $n$. The natural projection $C^n\to C^{(n)}$ gives $C^{(n)}$ the structure of an irreducible smooth projective variety of dimension $n$. The elements of $C^{(n)}$ can be considered as effective divisors of degree $n$ on $C$.

Consider the morphism $\rho:C^{(n)}\to\Pic^n\left(C\right)$, $D\mapsto\mathscr{O}_C\left(D\right)$. The image $W_n:=\rho\left(C^{(n)}\right)$ is a projective and irreducible subvariety of $\Pic^n\left(C\right)$ of dimension $n$ for $1\leq n\leq g$. Note that $W_n$ is the subset of $\Pic^n\left(C\right)$ of line bundles with a non-empty linear system, and for each line bundle $L\in\Pic^n\left(C\right)$, the fiber $\rho^{-1}\left(L\right)$ is the complete linear system $|L|$ by \cite[Abel's Theorem]{ACGH}. For a fixed divisor $E\in\Div^n\left(C\right)$ we consider the morphism $\alpha:C^{(n)}\to J\left(C\right)$, $D\mapsto\mathscr{O}_C\left(D-E\right)$ which is called the \textit{Abel-Jacobi map}.

Roughly speaking the \textit{``Generalized Torelli Problem''} consists of recovering information from a curve $C\in\mathscr{M}_g$ and families of linear systems on $C$ through the geometry of the Abel-Jacobi map and the varieties $W_n$. Connections between linear systems of a curve and its variety $W_n$ have been studied previously. We will mention some papers about it. 

First of all, given a projective curve $C\in\mathscr{M}_g$, it is known that its complete linear systems of positive dimension and degree $n\leq g-1$ are related to the singular locus of $W_n$ through the well-known \cite[Riemann-Kempf Singularity Theorem]{K}. This theorem was first proven by Riemann in $1865$ for
the case $n=g-1$, and then generalized by Kempf in $1973$ for $1\leq n\leq g-1$. Se also \cite[Riemann's Singularity Theorem]{GH} for the result over $\C$. Relations between linear systems of a curve $C$ and the smooth locus of $W_n$ have also been studied. For example, if we restrict the Generalized Torelli Problem to the case over $\C$ of complete pencils of linearly equivalent divisors and assuming that $W_n$ is smooth, in \cite{STR1} there is a result under these conditions that we will now state. To do this, we need a definition.

Consider the canonical morphism $\phi:C\to\mathbb{P}^{g-1}$, i.e., the morphism associated to a canonical divisor in $C$. For each divisor $D$ in the symmetric product $C^{(n)}$, we define the \textit{linear span of} $D$, denoted by $\overline{\phi\left(D\right) }$, as the intersection of all the hyperplanes $H$ in $\mathbb{P}^{g-1}$ such that $D\leq\phi^*H$. We define the map $\overline{\phi}:C^{(n)}\dasharrow\mathbb{G}\left(n-1,g-1\right)$, by $D\mapsto \overline{\phi\left(D\right)}$. By \cite[Riemann-Kempf Singularity Theorem]{K} and the Geometric Riemann-Roch Theorem \cite[p. 12]{ACGH} this map is well defined on the open subset of $C^{(n)}$ which consists of the divisors $D\in C^{(n)}$ such that $\ell\left(D\right)=1$. It's a known fact that the projectivized differential $\alpha_*$ of the Abel-Jacobi map is equal to $\overline{\phi}$.

Now we state the main result of \cite{STR1} mentioned above. For the notion of multiple locus of a morphism, see Chapter \ref{lm}. We will write $\mathfrak{g}_d^r$ for a \textit{linear system} of dimension $r$ and degree $d$, i.e., the projectivization of a vector subspace of dimension $r+1$ of $L\left(D\right)$, where $D\in\Div^d\left(C\right)$.

\begin{thm}\label{i1}
	Let $C\in\mathscr{M}_g$ be a smooth projective curve defined over $\C$. Assume that $C$ has a $\mathfrak{g}_{n+1}^1$, but not a $\mathfrak{g}_n^1$ nor a $\mathfrak{g}_{n+2} ^2$, for some $n$ such that $n\leq g-2$. Consider the map $\overline{\phi}:C^{(n)}\to\mathbb{G}\left(n-1,g-1\right)$, $R$ its multiple locus, and $B=\overline{\phi}\left(R\right)$, then the restriction $\overline{\phi}:R \to B$ is a ``universal'' $\mathfrak{g}_{n+1}^1$ for $C$.
\end{thm}

Note that the hypothesis that the curve $C$ doesn't have a $\mathfrak{g}_n^1$ in this theorem implies that $C$ is non-hyperelliptic and that $\ell\left(D\right)=1$ for all $D\in C^{(n)}$ (see Theorem \ref{gwn1}), so in this case the map $\overline{\phi}$ is defined on all $C^{(n)}$.

The case over $\C$ of complete nets of linearly equivalent divisors has been studied in the paper \cite{STR2}, where the same authors prove the following result. Here we consider the \textit{Gauss map} $\mathcal{G}:W_n\dashrightarrow\mathbb{G}\left(n-1,g-1\right)$ defined as $\mathcal{G}\left(\mathscr{O }_C\left(D\right)\right)=\overline{\phi\left(D\right)}=\overline{\phi}\left(D\right)$ for $\mathscr{O}_C\left(D\right)\in\left(W_n\right)_\sm$.

\begin{thm}\label{i2}
	Let $C\in\mathscr{M}_g$ be a complex non-hyperelliptic curve with $g\geq5$. Assume that $C$ has a $\mathfrak{g}_{n+2}^2$, with $3\leq n\leq g-2$, such that the plane model $\Gamma\subset\mathbb{P} ^2$ determined by the $\mathfrak{g}_{n+2}^2$ is birational to $C$ and is non-singular or only has nodes as singularities. If $\mathcal{G}:W_n\dashrightarrow\mathbb{G}\left(n-1,g-1\right)$ is the Gauss map on $W_n$, $R$ is its multiple locus and $B= \mathcal{G}\left(R\right)$, then the dual curve $\Gamma^*$ and the linear system $\mathfrak{g}_{n+2}^2$ can be reconstructed from the branch locus of the restriction of the map $\mathcal{G}$ to $R$, i.e., from $\mathcal{G}:R\to B$.
\end{thm}

In this theorem reconstructing means that $B$ contains a copy of the $\mathfrak{g}_{n+2}^2$ and a birational copy of the dual curve $\Gamma^*$. There was an attempt to generalize Theorem \ref{i1} by Faucette in his notes \cite{F} for the case of a complex non-hyperelliptic curve. In these notes Faucette tried to reconstruct all the complete linear systems $\mathfrak{g}_d^r$ of a curve $C\in\mathscr{M}_g$, for integers $r,d$ such that $d\leq g $, through the Gauss map on $W_{d-r}$, under some hypothesis of smoothness of $W_{d-r}$ and non-existence of $\mathfrak{g}_{d+1}^{r+1 }$'s, but his notes have a mistake. He identifies some algebraic curves with certain analytic $r$-folds, which cannot be for dimension reasons if $r\geq2$. On the other hand, the smoothness hypothesis of $W_{d-r}$ implies that the general curve $C\in\mathscr{M}_g$ doesn't have a $\mathfrak{g}_d^r$ if $r\geq3$ by Brill-Noether theory, and in fact it isn't clear if there exists a curve that satisfies these conditions. Moreover, the two cases $r\in\{1,2\}$ have been studied by Smith and Tapia-Recillas in the papers cited above. Now we state the main results of this work.

\subsection*{Non-hyperelliptic case}

Following this line, in Subsection \ref{tg} we will prove a certain generalization of Theorems \ref{i1} and \ref{i2} for a non-hyperelliptic curve defined over an arbitrary algebraically closed field $\mathbb{F}$ of characteristic zero. Assuming that $C$ doesn't have a $\mathfrak{g}_{n+k+1}^{k+1}$, we reconstruct all the complete $\mathfrak{g}_{n+k}^k$'s (where $1\leq k\leq n-1$) and the dual hypersurfaces of the images of its associated morphisms, via information encoded in the geometry of the Gauss map on $W_n$. We will now state the precise result that also appears in Subsection \ref{tg}. In the following $\mathcal{W}_n$ denotes $\rho^{-1}\left(\left(W_n\right)_\sm\right)$. We observe that $\mathcal{W}_n$ is isomorphic to $\left(W_n\right)_\sm$ vía $\rho$, and we will use this identification to define the \textit{Gauss map} on $\mathcal{W}_n$, i.e., $\mathcal{G}\left(D\right)=\overline{\phi\left(D\right)}$ for $D\in\mathcal{W}_n$.

\begin{thm}\label{i}
Let $1\leq k\leq n-1\leq g-3$ and $C\in\mathscr{M}_g$ be a non-hyperelliptic curve. Assume that $C$ has a $\mathfrak{g}_{n+k}^k$, but not a $\mathfrak{g}_{n+k+1}^{k+1}$. Now let $\mathcal{G}:\mathcal{W}_n\to\mathbb{G}\left(n-1,g-1\right)$ be the Gauss map on $\mathcal{W}_n$. Let $L$ be an arbitrary $\mathfrak{g}_{n+k}^k$ in $C$ and consider the morphism $\phi_L:C\to\mathbb{P}^k$ associated to $L$. Then the linear system $L$ and the dual hypersurface $\phi_L\left(C\right)^*$ can be reconstructed via the restriction of the Gauss map to the $(n+k)$-intersection locus $R_{n,k}$, i.e., from $\mathcal{G}:R_{n,k}\to\mathcal{G}\left(R_{n,k}\right)$.
\end{thm}

In this theorem we remove the hypothesis of smoothness of $W_n$ in order to ensure the existence of the linear systems in question, using Brill-Noether Theory. In this theorem the notion of the $\left(n+k\right)$-intersection locus plays a central role (see Chapter \ref{lmi} for the definition of this locus). These loci turn out to be proper closed subsets of $\mathcal{W}_n$ (Proposition \ref{lmi2}), which is due to the fact that the generic fiber of the Gauss map has only one element (Theorem \ref{g1}(c)). To understand a motivation for defining this locus see Proposition \ref{g4}. To reconstruct these dual hypersurfaces we use Bertini's Theorem.

\subsection*{Hyperelliptic case}

We study an analogous result for the hyperelliptic case. Consider the set $\mathcal{L}_{n+k}^k\left(C\right):=\{L:L\text{ is a complete }\mathfrak{g}_{n+k}^k \text{ in }C\}$. In the case that $C$ is hyperelliptic the image of the Gauss map $\mathcal{G}\left(\mathcal{W}_n\right)$ also contains information of the linear systems $L\in\mathcal{L} _{n+k}^k\left(C\right)$. In this case we obtained a geometric description for the image $\mathcal{G}\left(\mathcal{W}_n\right)$ that depends on the linear systems in $C$. We also show that $\overline{\mathcal{G}\left(\mathcal{W}_n\right)}$ contains copies of the dual hypersurfaces of the images of the morphisms associated to these linear systems. To state the precise result we need to define $\beta:\bigcup_{L\in\mathcal{L}_{n+k}^k\left(C\right)}L\to\mathbb{G}\left( n-1,g-1\right)$ by $F\mapsto\overline{\phi\left(F\right)}$. We will now state the precise result that also appears in Subsection \ref{tg0}.

\begin{thm}\label{2}
Let $C\in\mathscr{M}_g$ be hyperelliptic, and $1\leq k\leq n\leq g-1$. Consider the Gauss map $\mathcal{G}:\mathcal{W}_n\to\mathbb{G}\left(n-1,g-1\right)$ on $\mathcal{W}_n$. Then $\mathcal{G}\left(\mathcal{W}_n\right)$ is a union of sets whose closures are projective varieties birational to $\mathbb{P}^k$. Moreover, given $L\in\mathcal{L}_{n+k}^k\left(C\right)$, consider the morphism $\phi_L:C\to\mathbb{P}^k$ associated to $L$. Then $\beta\left(L\right)\subset\overline{\mathcal{G}\left(\mathcal{W}_n\right)}$ contains a copy of the dual hypersurface $\phi_L\left(C \right)^*$.
\end{thm}

To prove this theorem we use a characterization of the complete and special linear systems on hyperelliptic curves found in \cite[I. D-9]{ACGH}. Applying this theorem in the particular case $k=n$, we obtain that $\overline{\mathcal{G}\left(\mathcal{W}_n\right)}$ is a rational variety of dimension $n$.

\begin{cor}
	Let $C\in\mathscr{M}_g$ be a hyperelliptic curve, and $1\leq n\leq g-1$. Let $L$ be the unique complete $\mathfrak{g}_{2n}^n$ in $C$. Then $\overline{\mathcal{G}\left(\mathcal{W}_n\right)}=\beta\left(L\right)$ is birational to $\mathbb{P}^n$.
\end{cor}

Now we present a brief summary of the main results of each section: The second section is a short compendium of results on the geometry of $W_n$ (Theorem \ref{gwn1}). In the third section we study the cardinality of the fibers of the Gauss map. In particular, for a non-hyperelliptic curve $C\in\mathscr{M}_g$ and $2\leq n\leq g-2$, we prove that the generic fiber of the Gauss map on $W_n$ has one element (Theorem \ref{g1}). In the fourth section we characterize the multiple locus of the Gauss map (Theorem \ref{lm1}). In the fifth section we introduce the notion of larger intersection loci and we find some bounds for its dimensions (Proposition \ref{lmi4}). In the sixth and final section we use these loci to study the Generalized Torelli Problem in the non-hyperelliptic case, i.e., we prove Theorem \ref{i}. Then we study an analogous result for the hyperelliptic case (Theorem \ref{2}).

\section{Geometry of $W_n$}

This short section is a compendium of results on the geometry of $W_n$, which is summarized in the following theorem. We will use the notation $W_{-1}:=\varnothing$ and $W_0:=\{*\}$ a singleton.

\begin{thm}\label{gwn1}

	Let $1\leq n\leq g-1$. We have the following:
	\begin{itemize}
		\item[(a)] For each $C\in\mathscr{M}_g$ we have that $W_n$ is smooth if and only if $C$ doesn't have a $\mathfrak{g}_n^1$.
		\item[(b)] If $C\in\mathscr{M}_g$ is hyperelliptic, then $\Sing\left(W_n\right)$ is birational to $W_{n-2}$. In particular $\Sing\left(W_n\right)$ is irreducible and
		$$
		\dim\left(\Sing\left(W_n\right)\right)=n-2.
		$$
		\item[(c)] If $n\geq2$ and if $C\in\mathscr{M}_g$ is such that $W_n$ is smooth, then $C$ is non-hyperelliptic.
		\item[(d)] If $n\geq\frac{g}{2}+1$, then for each $C\in\mathscr{M}_g$ we have that $W_n$ is singular.
		\item[(e)] If $n<\frac{g}{2}+1$, then for general $C\in\mathscr{M}_g$ we have that $W_n$ is smooth.
	\end{itemize}
\end{thm}

In what follows of this section we will prove Theorem \ref{gwn1}.

\begin{proof}[Proof of (a)]
	Observe that $\Sing\left(W_n\right)\neq\varnothing$ if and only if there exists $\mathscr{O}_C\left(D\right)\in\Sing\left(W_n\right)$, if and only if $\ell\left(D\right)\geq2$ (by the \cite[Riemann-Kempf Singularity Theorem]{K}), if and only if $C$ has a $\mathfrak{g}_n^1$. This proves item (a).
\end{proof}

\begin{proof}[Proof of (b)]
In this item we generalize the argument of \cite[Proposition 11.2.8.]{BL} for $1\leq n\leq g-1$.

If $n=1$, then $W_1\simeq C$ is smooth, so
	$$
	\Sing\left(W_1\right)=\varnothing=W_{-1},
	$$
	so that
	$$
	\dim\left(\Sing\left(W_1\right)\right)=-1=1-2.
	$$
	
	This proves the case $n=1$.
	
	Let's now assume that $n=2$. Let $p\in C$. We claim that
	$$
	\Sing\left(W_2\right)=\{\mathscr{O}_C\left(p+\iota\left(p\right)\right)\},
	$$
	where $\iota$ is the hyperelliptic involution.
	
	First, it is clear that $\mathscr{O}_C\left(p+\iota\left(p\right)\right)\in\Sing\left(W_2\right)$, since $\ell\left(p+\iota\left(p\right)\right)=2$.
	
	Conversely, let $\mathscr{O}_C\left(p_1+p_2\right)\in\Sing\left(W_n\right)$, i.e., $\ell\left(p_1+p_2\right)\geq2$ by \cite[Riemann's Singularity Theorem]{GH}. We must have that $\ell\left(p_1+p_2\right)=2$, because if $\ell\left(p_1+p_2\right)\geq3$, we would have $\ell\left(p_1\right) \geq2$, which would imply $C\simeq\mathbb{P}^1$, which is a contradiction. Thus, we see that $|p_1+p_2|$ is the unique $\mathfrak{g}_2^1$ in $C$, since $C$ is hyperelliptic. Then $p_1+p_2\sim p+\iota\left(p\right)$, which implies $\mathscr{O}_C\left(p_1+p_2\right)\simeq\mathscr{O}_C\left( p+\iota\left(p\right)\right)$. Therefore
	$$
	\Sing\left(W_2\right)=\{\mathscr{O}_C\left(p+\iota\left(p\right)\right)\},
	$$
	so
	$$
	\dim\left(\Sing\left(W_2\right)\right)=0=2-2,
	$$
	which proves the case $n=2$.
	
	Now, let's assume that $n\geq3$. Let $p\in C$ and consider the morphism
	\begin{equation*}
		\begin{split}
			\Phi:W_{n-2}\to\Sing\left(W_{n-2}\right), \
			\mathscr{O}_C\left(D\right)\mapsto\mathscr{O}_C\left(D\right)\otimes\mathscr{O}_C\left(p+\iota\left(p\right)\right).
		\end{split}
	\end{equation*}
	
	We will prove that $\Phi$ is birational.

	Since $\Phi$ is proper, it is enough to prove that $\Phi$ is bijective and use \cite[6, Lemma 2.4]{SH}, \cite[Proposition B.1]{FS}. First we will prove that $\Phi$ is surjective. Let $\mathscr{O}_C\left(D\right)\in\Sing\left(W_{n-2}\right)$, i.e., $\ell\left(D\right)\geq2$ by \cite[Riemann's Singularity Theorem]{GH}. Then $|D|$ is a complete $\mathfrak{g}_n^{\ell\left(D\right)-1}$. Moreover, since $D$ is a special divisor, by \cite[I. D-9]{ACGH} the divisor $D$ must be of the form
	$$
	D=q+\iota\left(q\right)+p_1+\dots+p_{n-2},
	$$
	for some $q,p_1,\dots,p_{n-2}\in C$. This implies that
	\begin{equation*}
		\begin{split}
			\mathscr{O}_C\left(D\right)&=\mathscr{O}_C\left(q+\iota\left(q\right)+p_1+\dots+p_{n-2}\right)\\
			&\simeq\mathscr{O}_C\left(p_1+\dots+p_{n-2}\right)\otimes\mathscr{O}_C\left(q+\iota\left(q\right)\right)\\
			&=\Phi\left(\mathscr{O}_C\left(p_1+\dots+p_{n-2}\right)\right)\in\Phi\left(W_{n-2}\right),
		\end{split}
	\end{equation*}
	which proves that $\Phi$ is surjective.

	Now we will prove that $\Phi$ is injective. Let's assume that
	$$
	\Phi\left(\mathscr{O}_C\left(D\right)\right)\simeq\Phi\left(\mathscr{O}_C\left(E\right)\right),
	$$
	i.e.,
	$$
	\mathscr{O}_C\left(D\right)\otimes\mathscr{O}_C\left(p+\iota\left(p\right)\right)\simeq\mathscr{O}_C\left(E\right)\otimes\mathscr{O}_C\left(p+\iota\left(p\right)\right),
	$$
	which implies
	$$
	\mathscr{O}_C\left(D\right)\simeq\mathscr{O}_C\left(E\right),
	$$
	so $\Phi$ is injective.

	This shows that $\Phi$ is bijective and therefore birational. Then, in particular $\Sing\left(W_n\right)$ is irreducible and
	$$
	\dim\left(\Sing\left(W_n\right)\right)=\dim\left(W_{n-2}\right)=n-2.
	$$
	
	This proves item (b).
\end{proof}

\begin{proof}[Proof of (c)]
	It's clear by item (b).	
\end{proof}

\begin{proof}[Proof of (d) and (e)]
Recall the definition of the \textit{Brill-Noether Number} 
$$
\rho\left(g,r,d\right):=g-\left(r+1\right)\left(g-d+r\right).
$$ 

Items (d) and (e) are obtained directly from item (b) and from Brill-Noether Theory found in \cite[V. (1.1)]{ACGH}, from where it is obtained that
	\begin{itemize}
		\item if $\rho\left(g,1,n\right)\geq0$, then each $C\in\mathscr{M}_g$ has a $\mathfrak{g}_n^1$;
		\item if $\rho\left(g,1,n\right)<0$, then the general curve $C\in\mathscr{M}_g$ doesn't have a $\mathfrak{g}_n^1$.
		
	\end{itemize}
	Then, the results are obtained by the fact
	
	$$
	\rho\left(g,1,n\right)\geq0\Leftrightarrow n\geq\tfrac{g}{2}+1.
	$$
	
	This proves items (d) and (e).
\end{proof}

\section{Cardinality of the fibers of the Gauss map}

This section is divided into two subsections. In the first we determine the cardinality of the generic fiber of the Gauss map for an arbitrary curve $C\in\mathscr{M}_g$ and for each $1\leq n\leq g-1$. We remind the reader the well-known cases: the hyperelliptic case with $1\leq n\leq g-1$, and the non-hypereliptic case with $n\in\{1,g-1\}$. In the non-hyperelliptic case with $2\leq n\leq g-2$ we prove that the generic fiber of the Gauss map has only one element. In the second subsection we study conditions under which the fibers of the Gauss map have lower cardinality than the generic fiber in the hyperelliptic cases with $1\leq n\leq g-1$, and non-hyperelliptic with $n=g-1$. We develop an example of the non-hyperelliptic case with $g=4$ and $n=3$. 

\subsection{Cardinality of the generic fiber}\label{ck}

In this subsection we will determine the cardinality of the generic fiber of the Gauss map for an arbitrary curve $C\in\mathscr{M}_g$ and $1\leq n\leq g-1$, which is summarized in the following result. We will write $\mathcal{G}_n$ when we want to emphasize that the Gauss map is defined on $\mathcal{W}_n$. When this is understood we will simply write down $\mathcal{G}$.

\begin{thm}\label{g1}
	Let $C\in\mathscr{M}_g$ be a curve and $1\leq n\leq g-1$. We have the following:
	\begin{itemize}
		\item[(a)] If $C$ is hyperelliptic, then the generic fiber of $\mathcal{G}_n$ has $2^n$ elements.
		\item[(b)] If $C$ is non-hyperelliptic, then the generic fiber of $\mathcal{G}_{g-1}$ has $\binom{2g-2}{g-1}$ elements .
		\item[(c)] If $C$ is non-hyperelliptic and $1\leq n\leq g-2$, then the generic fiber of $\mathcal{G}_n$ has one element.
	\end{itemize}
\end{thm}

In what follows our goal will be to prove Theorem \ref{g1}. We will begin by proving that the Gauss map doesn't have fibers of positive dimension.

\begin{prop}
	For $1\leq n\leq g-1$ each fiber of the Gauss map $\mathcal{G}$ is a finite set.
\end{prop}

\begin{proof}
	
	Recall that $\phi$ is the canonical map of $C$.	Let $W\in\mathbb{G}\left(n-1,g-1\right)$. We will prove that $\mathcal{G}^{-1}\left(W\right)$ is finite. If $\mathcal{G}^{-1}\left(W\right)=\varnothing$, it is finite. Let's then assume that $\mathcal{G}^{-1}\left(W\right)\neq\varnothing$. 
	
	We claim that $W\cap\phi\left(C\right)$ is finite. Indeed, taking a hyperplane $H$ in $\mathbb{P}^{g-1}$ such that $W\subset H$, we have $W\cap\phi\left(C\right)\subset H \cap\phi\left(C\right)$. Since $\deg\left(\phi\left(C\right)\right)\in\{g-1,2g-2\}$, we have
	$$
	\#\left(W\cap\phi\left(C\right)\right)\leq\#\left(H\cap\phi\left(C\right)\right)\leq2g-2,
	$$
	so $W\cap\phi\left(C\right)$ is finite.
	
	Then we have
	$$
	W\cap\phi\left(C\right)=\{\phi\left(p_1\right),\dots,\phi\left(p_k\right)\}
	$$
	for some $k$. Now since
	$$
	\mathcal{G}^{-1}\left(W\right)\subset\{q_1+\dots+q_n\in C^{(n)}:\phi\left(q_j\right)\in\{\phi\left(p_1\right),\dots,\phi\left(p_k\right)\}\}
	$$
	and this last set is finite, then $\mathcal{G}^{-1}\left(W\right)$ is finite too.
\end{proof}

As we said before, items (a) and (b) of Theorem \ref{g1} are well-known, so in what follows we will prove item (c) of this theorem. To do this we will need some previous results. For $W\in\mathbb{G}\left(n-1,g-1\right)$ we define the \textit{intersection divisor} $\left(W\cdot C\right)$ as the greatest common divisor of the divisors $ \{H\cap C:H\text{ hyperplane in }\mathbb{P}^{g-1}\text{ such that }H\supset W\}$.

\begin{prop}\label{g2}
	Let $C\in\mathscr{M}_g$ be non-hyperelliptic, $1\leq n\leq g-2$, and let $K$ be a canonical divisor on $C$. Then there exists $E\in\mathcal{W}_n$ such that $|K-E|$ is base point free.
\end{prop}

\begin{proof}
	Let $H$ be a general hyperplane in $\mathbb{P}^{g-1}$ so that $K=\left(H\cdot C\right)$. Since $|K|$ is base point free, by Bertini's Theorem we can assume that $K$ is reduced. Since $\deg\left(C\right)=2g-2$, we have
	$$
	H\cap C=\{p_1,\dots,p_{2g-2}\},
	$$
	where the $p_1,\dots,p_{2g-2}\in C$ are different, and by the General Position Theorem \cite[p. 109]{ACGH} each subset of $k\leq g-1$ elements of $H\cap C$ is linearly independent. Let $E:=p_1+\dots+p_n$. Since $n\leq g-2$, then $\{p_1,\dots,p_n\}$ is linearly independent, so by the Geometric Riemann-Roch Theorem
	\begin{equation*}
		\begin{split}
			\ell\left(E\right)&=\deg\left(E\right)-\dim\left(\overline{E}\right)\\
			&=n-\left(n-1\right)=1.
		\end{split}
	\end{equation*}
	
	By the Riemann-Roch Theorem
	\begin{equation*}
		\begin{split}
			\ell\left(K-E\right)&=\ell\left(E\right)-\deg\left(E\right)+g-1\\
			&=1-n+g-1=g-n.
		\end{split}
	\end{equation*}
	
	We claim that $|K-E|=|p_{n+1}+\dots+p_{2g-2}|$ is base point free. We will prove it by contradiction. Assume that $|K-E|$ has a base point, say $p\in C$. Then, $p\leq p_{n+1}+\dots+p_{2g-2}$, so, without loss of generality we can assume that $p=p_{n+1}$. Since $p$ is a base point of $|K-E|$, we have
	$$
	\ell\left(K-E-p\right)=\ell\left(K-E\right)=g-n,
	$$
	so, by the Riemann-Roch theorem
	\begin{equation*}
		\begin{split}
			\ell\left(E+p\right)&=\ell\left(K-E-p\right)+\deg\left(E+p\right)-g+1\\
			&=\left(g-n\right)+\left(n+1\right)-g+1=2.
		\end{split}
	\end{equation*}
	
	Then, by the Geometric Riemann-Roch theorem
	\begin{equation*}
		\begin{split}
			\dim\left(\overline{E+p}\right)&=\deg\left(E+p\right)-\ell\left(E+p\right)\\
			&=\left(n+1\right)-2=n-1,
		\end{split}
	\end{equation*}
	so $\{p_1,\dots,p_{n+1}\}$ is linearly dependent. On the other hand, this set is linearly independent, since $n+1\leq g-1$, which is a contradiction. Therefore $|K-E|$ must be base point free.
\end{proof}

For each $r\geq0$, $d\geq1$ the set
$$
C_d^r:=\{D\in C^{(d)}:\dim|D|\geq r\}
$$
is a closed subvariety of $C^{(d)}$ (see \cite[IV. \textsection3]{ACGH}).

For each $k\geq1$ consider the diagram
$$
\begin{tikzcd}
	C^{(n)}\times C^{(k)} \arrow[swap]{d}{p} \arrow{r}{s} & C^{(n+k)} \\
	C^{(n)}
\end{tikzcd}
$$
where the morphisms $s$ and $p$ are given by the sum and the projection respectively. The set $s^{-1}\left(C_{n+k}^k\right)$ is a proper closed subset of $C^{(n)}\times C^{(k)}$. Since $p$ is a proper morphism, then the set $\mathcal{C}_{n,k}:=p\left(s^{-1}\left(C_{n+k}^k\right)\right)$ is a proper closed subset of $C^{(n)}$. We denote $\mathcal{C}_n:=\mathcal{C}_{n,1}$.

\begin{prop}\label{g3}
	Let $C\in\mathscr{M}_g$ be non-hyperelliptic and $1\leq n\leq g-2$. Then the set $\mathcal{U}_n:=\mathcal{W}_n\setminus\mathcal{C}_n$ is non-empty. Therefore $\mathcal{U}_n$ is an open (dense) subset of $\mathcal{W}_n$.
\end{prop}

\begin{proof}
	If $C_{n+1}^1=\varnothing$, we have $\mathcal{C}:=p\left(s^{-1}\left(\varnothing\right)\right)=\varnothing$ , so $\mathcal{U}_n:=\mathcal{W}_n\neq\varnothing$. Suppose then that $C_{n+1}^1\neq\varnothing$. Then $\mathcal{C}_n:=p\left(s^{-1}\left(C_{n+1}^1\right)\right)\neq\varnothing$. We observe that $C_n^1\subset\mathcal{C}$. Indeed, if $C_n^1=\varnothing$, clearly we have the inclusion. Let's then assume that $C_n^1\neq\varnothing$. Let $D\in C_n^1$, i.e., $\dim|D|\geq1$. Then, given $q\in C$ we have $\dim|D+q|\geq\dim|D|\geq1$, so $s\left(D,q\right)=D+q\in C_ {n+1}^1$, i.e., $(D,q)\in s^{-1}\left(C_{n+1}^1\right)$. Therefore $D=p\left(D,q\right)\in p\left(s^{-1}\left(C_{n+1}^1\right)\right)=:\mathcal{C }$, which proves the inclusion. This implies that $\mathcal{C}_n^c\subset\left(C_n^1\right)^c=\mathcal{W}_n$. Note also that $\mathcal{C}_n\setminus C_n^1\neq\varnothing$ if and only if there exist $D\in\mathcal{W}_n$ and $q\in C$ such that $\ell\left( D+q\right)=2$.
	
	Now we will prove that $\mathcal{U}_n:=\mathcal{W}_n\setminus\mathcal{C}_n\neq\varnothing$. By Proposition \ref{g2} there exists $E\in\mathcal{W}_n$ such that $|K-E|$ is base point free. We will prove that $E\notin\mathcal{C}_n$. Note that, by the Riemann-Roch Theorem
	\begin{equation*}
		\begin{split}
			\ell\left(K-E\right)&=\ell\left(E\right)-\deg\left(E\right)+g-1\\
			&=1-n+g-1=g-n.
		\end{split}
	\end{equation*}
	
	Now, by contradiction, suppose that $E\in\mathcal{C}$. Then, there exists $q\in C$ such that $E=p\left(E,q\right)$, with $\left(E,q\right)\in s^{-1}\left(C_{ n+1}^1\right)$, so that $E+q=s\left(E,q\right)\in C_{n+1}^1$, i.e., $\dim|E+q |\geq1$, but since $\dim|E|=0$, then we must have $\dim|E+q|=1$. By the Riemann-Roch theorem
	\begin{equation*}
		\begin{split}
			\ell\left(K-E-q\right)&=\ell\left(E+q\right)-\deg\left(E+q\right)+g-1\\
			&=2-\left(n+1\right)+g-1=g-n,
		\end{split}
	\end{equation*}
	so $\ell\left(K-E-q\right)=\ell\left(K-E\right)$. This implies that $L(K-E-q)=L(K-E)$, i.e., $q$ is a base point of $|K-E|$, which is a contradiction, since $|K-E|$ is base point free. Therefore we must have that $E\notin\mathcal{C}_n$, so $\mathcal{U}_n$ is a non-empty open (dense) subset of $\mathcal{W}_n$.
\end{proof}

To prove item (c) of Theorem \ref{g1} we will also need the following result.

\begin{prop}\label{g4}
	Let $C\in\mathscr{M}_g$ be non-hyperelliptic and $1\leq n\leq g-2$. Then there exists an open (dense) subset $\mathcal{U}_n$ of $\mathcal{W}_n$ such that $\left(\overline{D}\cdot C\right)=D$ for all $D\in\mathcal{U}_n$.
\end{prop}

\begin{proof}
	Let $\mathcal{U}_n:=\mathcal{W}_n\setminus\mathcal{C}_n$ be the open (dense) subset of $\mathcal{W}_n$ of Proposition \ref{g3}. Let $D\in\mathcal{U}_n$.

	We will prove that $\left(\overline{D}\cdot C\right)=D$. Indeed, as $D\leq\left(\overline{D}\cdot C\right)$, we have $\deg\left(\overline{D}\cdot C\right)\geq n$. It is enough to prove $\deg\left(\overline{D}\cdot C\right)=n$. By contradiction, suppose that $\deg\left(\overline{D}\cdot C\right)\geq n+1$. Then, there exists $q\in C$ such that $D\leq D+q\leq\left(\overline{D}\cdot C\right)$. Note that $\overline{D+q}=\overline{D}\in\mathbb{G}\left(n-1,g-1\right)$. Then, by the Geometric Riemann-Roch Theorem
	\begin{equation*}
		\begin{split}
			\ell\left(D+q\right)&=\deg\left(D+q\right)-\dim\left(\overline{D+q}\right)\\
			&=\left(n+1\right)-\left(n-1\right)=2,
		\end{split}
	\end{equation*}
	i.e., $\dim|D+q|=1$, so that $s\left(D,q\right)=D+q\in C_{n+1}^1$. Then $(D,q)\in s^{-1}\left(C_{n+1}^1\right)$, which implies $D=p\left(D,q\right)\in p\left(s^{-1}\left(C_{n+1}^1\right)\right)=:\mathcal{C}_n$, but this contradicts the fact that $D\in\mathcal{U}_n$ . Therefore we must have that $\deg\left(\overline{D}\cdot C\right)=n$, so that $\left(\overline{D}\cdot C\right)=D$.
\end{proof}

Now we will finish the proof of Theorem \ref{g1} by proving item (c).

\begin{proof}[Proof of (c)]
	If $n=1$, then $\mathcal{G}=\phi$ is injective, so each fiber of $\mathcal{G}$ has one element. 
	
	Assume $n\geq2$. Let $\mathcal{U}_n:=\mathcal{W}_n\setminus\mathcal{C}_n$ be the open (dense) subset of $\mathcal{W}_n$ of Proposition \ref{g3}. Note that $\mathcal{G}\left(\mathcal{U}_n\right)$ is a constructible set, so it contains a dense open $O_n$ of $\overline{\mathcal{G}\left(\mathcal{U}_n\right)}$ (see \cite[1.3, Proposition]{BO}). We will prove that $\mathcal{G}^{-1}\left(\overline{D}\right)=\{D\}$ for each $\overline{D}\in O_n$. Let $\overline{D}=\mathcal{G}\left(D\right)\in O_n$, with $D\in\mathcal{U}_n$. It is clear that $D\in\mathcal{G}^{-1}\left(\overline{D}\right)$.
	
	Conversely, let $E\in\mathcal{G}^{-1}\left(\overline{D}\right)$, i.e.,
	$$
	\overline{E}=\mathcal{G}\left(E\right)=\overline{D}
	$$
	
	Note that
	$$
	E\leq\left(\overline{E}\cdot C\right)=\left(\overline{D}\cdot C\right)=D,
	$$
	where the last equality is held by Proposition \ref{g4}. Then $D-E\geq0$ and $\deg\left(D-E\right)=0$, which implies $D-E=0$, i.e., $E=D$. This proves that $\mathcal{G}^{-1}\left(\overline{D}\right)=\{D\}$ for every $\overline{D}\in O_n$ and $O_n$ is a open dense subset of $\overline{\mathcal{G}\left(\mathcal{U}_n\right)}$. Thus the generic fiber of $\mathcal{G}$ has one element, which proves item (c) of Theorem \ref{g1}. This completes the proof of Theorem \ref{g1}.
\end{proof}

\subsection{Fibers of lower cardinality}

In this subsection we will analyze conditions under which the fibers of the Gauss map are of lower cardinality than the generic fiber. We will divide this study into two cases: the hyperelliptic case with $1\leq n\leq g-1$, and the non-hyperelliptic case with $n=g-1$. We will give an example of the non-hyperelliptic case with $g=4$ and $n=3$.

\subsubsection{Hyperelliptic case, $1\leq n\leq g-1$}\label{g0}

\begin{prop}\label{g6}
	Let $C\in\mathscr{M}_g$ be hyperelliptic and $1\leq n\leq g-1$. Let $D\in\mathcal{W}_n$, $W:=\mathcal{G}\left(D\right)$ and $q\in C$ such that $q\notin\Supp\left(D\right)$. Then $\phi\left(q\right)\in W$ if and only if $q=\iota\left(p_i\right)$ for some $p_i\in\Supp\left(D\right)$.
\end{prop}

\begin{proof}
	If $n=1$, then $D$ is a point of $C$, say $D=p\in C$ and $\mathcal{G}=\phi$. In this case $W=\{\phi\left(p\right)\}$, so since $q\neq p$, we have that $\phi\left(q\right)=\phi\left( p\right)$ if and only if $q=\iota\left(p\right)$.
	
	Let's then assume $n\geq2$. Let's write $D=p_1+\dots+p_n$, $W=H_1\cap\dots\cap H_{g-n}$ where the $H_j$ are hyperplanes in $\mathbb{P}^{g-1}$.
	
	Let's first assume that $\phi\left(q\right)\in W$. We will prove that $q=\iota\left(p_j\right)$ for some $j\in\{1,\dots,n\}$. Since $W=\overline{\phi\left(D\right)}$ and $\phi\left(q\right)\in W$, for each $j\in\{1,\dots,g-n\} $ we have
	$$
	D=p_1+\dots+p_n\leq\phi^*H_j,
	$$
	$$
	q\leq\phi^*H_j,
	$$
	which implies
	$$
	D+q=p_1+\dots+p_n+q\leq\phi^*H_j
	$$
	for all $j\in\{1,\dots,g-n\}$. Then
	$$
	W\subset\overline{\phi\left(D+q\right)}\subset H_1\cap\dots\cap H_{g-n}=W,
	$$
	i.e., $\overline{\phi\left(D+q\right)}=W$. Now, by the Geometric Riemann-Roch Theorem
	\begin{equation*}
		\begin{split}
			\ell\left(D+q\right)&=\deg\left(D+q\right)-\dim\left(W\right)\\
			&=\left(n+1\right)-\left(n-1\right)=2,
		\end{split}
	\end{equation*}
	i.e., $|D+q|$ is a complete $\mathfrak{g}_{n+1}^1$. Since $D+q$ is a special divisor, by \cite[I. D-9]{ACGH} we have that
	$$
	|D+q|=\mathfrak{g}_2^1+P_1+\dots+P_{n-1},
	$$
	where no two of the $P_j\in C$ are conjugated under $\iota$. In particular, there must exist $p\in C$ such that
	$$
	p_1+\dots+p_n+q=D+q=p+\iota\left(p\right)+P_1+\dots+P_{n-1}.
	$$
	Since $D\in\mathcal{W}_n$, without loss of generality we can assume that $p_j=P_j$ for all $j\in\{1,\dots,n-1\}$, $p_n=p $, $q=\iota\left(p\right)$, which proves the first implication.
	
	Conversely, let us assume that $q=\iota\left(p_1\right)$. For each $j\in\{1,\dots,g-n\}$, let $L_j$ be a linear form such that $H_j=V(L_j)$. Then
	$$
	L\left(\phi\left(q\right)\right)=L_j\left(\phi\left(\iota\left(p_1\right)\right)\right)=L_j\left(\phi\left (p_1\right)\right)=0,
	$$
	for all $j\in\{1,\dots,g-n\}$, since $\phi\left(p_1\right)\in W=H_1\cap\dots\cap H_{g-n}$. Thus $\phi\left(q\right)\in H_1\cap\dots\cap H_{g-n}=W$.
	
	This proves the proposition.
\end{proof}

\begin{prop}
	Let $C\in\mathscr{M}_g$ be hyperelliptic and $1\leq n\leq g-1$. Let $D\in\mathcal{W}_n$, $W:=\mathcal{G}\left(D\right)$. Then $\#\left(\mathcal{G}^{-1}\left(W\right)\right)<2^n$ if and only if $D$ is non-reduced or some $p\in\Supp \left(D\right)$ is fixed by $\iota$.
\end{prop}

\begin{proof}
	If $n=1$, then $D$ is a point of $C$, say $D=p\in C$ and $\mathcal{G}=\phi$. In this case we have $\phi^{-1}\left(\phi\left(p\right)\right)=\{p,\iota\left(p\right)\}$, so $\#\left(\phi^{-1}\left(\phi\left(p\right)\right)\right)=1$ if and only if $p$ is a fixed point of $\iota$.
	
	Let's then assume that $n\geq2$, and write $D=p_1+\dots+p_n$.
	
	First we will prove that $W\cap\phi\left(C\right)=\{\phi\left(p_1\right),\dots,\phi\left(p_n\right)\}$. Indeed, on the one hand it is clear that $\{\phi\left(p_1\right),\dots,\phi\left(p_n\right)\}\subset W\cap\phi\left(C\right) $, since $W=\overline{\phi\left(D\right)}$. Conversely, let $\phi\left(q\right)\in W\cap\phi\left(C\right)$. If $q\in\Supp\left(D\right)$, it's done. Let's assume $q\notin\Supp\left(D\right)$. Then, by Proposition \ref{g6} we have $q=\iota\left(p_j\right)$ for some $j\in\{1,\dots,n\}$, so $\phi\left( q\right)=\phi\left(\iota\left(p_j\right)\right)=\phi\left(p_j\right)$, which proves the other inclusion.
	
	Now, suppose that $D$ is reduced and that no $p_j$ is a fixed point of $\iota$. By the equality
	$$
	\mathcal{G}^{-1}\left(W\right)=\{q_1+\dots+q_n\in\mathcal{W}_n:q_j\in\{p_j,\iota\left(p_j\right) \}\},
	$$
	it follows that $\#\left(\mathcal{G}^{-1}\left(W\right)\right)=2^n$. By contrapositive we obtain that if $\#\left(\mathcal{G}^{-1}\left(W\right)\right)<2^n$, then $D$ is non-reduced or some $p_j$ is a fixed point of $\iota$.
	
	Now we will prove the converse. Let's first assume that $D$ is non-reduced. Without loss of generality we can assume that $p_1=p_2$. We claim that
	$$
	\mathcal{G}^{-1}\left(W\right)\subset\{2q_2+q_3+\dots+q_n\in\mathcal{W}_n:q_j\in\{p_j,\iota\left(p_j \right)\}\}.
	$$
	
	Indeed, let $E=q_1+\dots+q_n\in\mathcal{G}^{-1}\left(W\right)$, i.e., $\overline{\phi\left(E\right)}= W=\overline{\phi\left(D\right)}$.
	
	As before, we can prove that $W\cap\phi\left(C\right)=\{\phi\left(q_1\right),\dots,\phi\left(q_n\right)\}$, so
	$$
	\{\phi\left(q_1\right),\dots,\phi\left(q_n\right)\}=\{\phi\left(p_1\right),\dots,\phi\left(p_n\right )\}.
	$$
	
	Furthermore, the intersection numbers of $W$ and $\phi\left(C\right)$ in the $\phi\left(p_j\right)$, which are the multiplicities of the $p_j$ in $D$ , must coincide with the intersection numbers of $W$ and $\phi\left(C\right)$ in the $\phi\left(q_i\right)$, which are the multiplicities of the $q_i$ in $E $. Then we must have that
	$$
	\left(\phi\left(q_{j_1}\right),\dots,\phi\left(q_{j_n}\right)\right)=\left(\phi\left(p_1\right),\dots ,\phi\left(p_n\right)\right),
	$$
	for some $j_i$ such that $\{j_1,\dots,j_n\}=\{1,\dots,n\}$. Relabeling if necessary, without loss of generality we can assume that
	$$
	\left(\phi\left(q_1\right),\dots,\phi\left(q_n\right)\right)=\left(\phi\left(p_1\right),\dots,\phi\left(p_n\right)\right),
	$$
	so $q_j\in\{p_j,\iota\left(p_j\right)\}$ for each $j\in\{1,\dots,n\}$. 
	
	We will prove that $q_1=q_2$. Suppose $q_1=p_1$. Then, since $p_1=p_2$ and $\ell\left(E\right)=1$, we must have that $q_2=p_2$, so $q_1=p_1=p_2=q_2$. Now suppose that $q_1=\iota\left(p_1\right)$. Then, since $p_1=p_2$ and $\ell\left(E\right)=1$, we must have that $q_2=\iota\left(p_2\right)$, so $q_1=\iota\left (p_1\right)=\iota\left(p_2\right)=q_2$. This proves the inclusion.
	
	Therefore
	$$
	\#\left(\mathcal{G}^{-1}\left(W\right)\right)\leq\#\{2q_2+q_3+\dots+q_n\in\mathcal{W}_n:q_j\in \{p_j,\iota\left(p_j\right)\}\}\leq2^{n-1}<2^n.
	$$
	
	Now suppose that some $p_j$ is a fixed point of $\iota$, say $\iota\left(p_1\right)=p_1$. We claim that
	$$
	\mathcal{G}^{-1}\left(W\right)\subset\{p_1+q_2+\dots+q_n\in\mathcal{W}_n:q_j\in\{p_j,\iota\left(p_j \right)\}\}.
	$$
	
	Indeed, given $E=q_1+\dots+q_n\in\mathcal{G}^{-1}\left(W\right)$, as in the previous case, without loss of generality we can assume that
	$$
	\left(\phi\left(q_1\right),\dots,\phi\left(q_n\right)\right)=\left(\phi\left(p_1\right),\dots,\phi\left( p_n\right)\right),
	$$
	so $q_j\in\{p_j,\iota\left(p_j\right)\}$ for all $j$. In particular $q_1\in\{p_1,\iota\left(p_1\right)\}=\{p_1\}$, which proves the claim.
	
	Therefore
	$$
	\#\left(\mathcal{G}^{-1}\left(W\right)\right)\leq\#\{p_1+q_2+\dots+q_n\in\mathcal{W}_n:q_j\in \{p_j,\iota\left(p_j\right)\}\}\leq2^{n-1}<2^n.
	$$
	
	This proves the proposition.
\end{proof}

\subsubsection{Non-hyperelliptic case, $n=g-1$}

\begin{prop}\label{g7}
	Let $C\in\mathscr{M}_g$ be non-hyperelliptic, $D\in\mathcal{W}_{g-1}$. Let $H:=\mathcal{G}\left(D\right)$, $K=\left(H\cdot C\right)$. If $K$ is non-reduced, then $\#\left(\mathcal{G}^{-1}\left(H\right)\right)<\binom{2g-2}{g-1}$.
\end{prop}

\begin{proof}
	We claim that
	$$
	\mathcal{G}^{-1}\left(H\right)\subset\{E\in\mathcal{W}_{g-1}:E\leq K\}.
	$$
	
	Indeed, let $E\in\mathcal{G}^{-1}\left(H\right)$. Then $\overline{E}=H$, so
	$$
	E\leq\left(\overline{E}\cdot C\right)=\left(H\cdot C\right)=K.
	$$
	which proves inclusion.
	
	Therefore, since $K$ is non-reduced, we have
	$$
	\#\left(\mathcal{G}^{-1}\left(H\right)\right)\leq\#\{E\in\mathcal{W}_{g-1}:E\leq K \}<\binom{2g-2}{g-1},
	$$
	since $\deg\left(K\right)=2g-2$. This proves the proposition.
\end{proof}

Next, through an example, we will see that the reciprocal implication of Proposition \ref{g7} doesn't hold. Let $C$ be a non-hyperelliptic curve of genus $4$. We will identify $C$ with its image under the canonical embedding. From \cite[IV, Example 5.5.2]{H} we know that $C$ has a $\mathfrak{g}_3^1$, say $L$. Note that $L$ is base point free, because if $L$ had a base point $p\in C$, then $L-p$ would be a $\mathfrak{g}_2^1$ in $C$, i.e., $C$ would be hyperelliptic, which it cannot be. Then, by Bertini's Theorem we can take a reduced $p+q+r\in L$, i.e., $p,q,r\in C$ are different.

We claim that $L=|p+q+r|$. Indeed, as $\ell\left(p+q+r\right)\geq2$, by the Geometric Riemann-Roch Theorem we have
\begin{equation*}
	\begin{split}
		1\leq\dim\left(\overline{p+q+r}\right)&=\deg\left(p+q+r\right)-\ell\left(p+q+r\right) \\
		&\leq3-2=1,
	\end{split}
\end{equation*}
so $\dim\left(\overline{p+q+r}\right)=1$. Using the Geometric Riemann-Roch Theorem again, we obtain
\begin{equation*}
	\begin{split}
		\ell\left(p+q+r\right)&=\deg\left(p+q+r\right)-\dim\left(\overline{p+q+r}\right)\\
		&=3-1=2,
	\end{split}
\end{equation*}
which proves $L=|p+q+r|$.

Now, by the Riemann-Roch Theorem \cite[IV, Theorem 1.3]{H}
\begin{equation*}
	\begin{split}
		\ell\left(K-p-q-r\right)&=\ell\left(p+q+r\right)-\deg\left(p+q+r\right)+4-1\\
		&=2-3+4-1=2,
	\end{split}
\end{equation*}
so $|K-p-q-r|$ is a complete $\mathfrak{g}_3^1$. As before $|K-p-q-r|$ must be base point free, because otherwise $C$ would have a $\mathfrak{g}_2^1$, i.e., $C$ would be hyperelliptic, which is not the case.

Then, in particular
$$
\ell\left(K-2p-q-r\right)=\ell\left(K-p-2q-r\right)=\ell\left(K-p-q-2r\right)=1,
$$
so
$$
|K-2p-q-r|+p\subsetneqq|K-p-q-r|,
$$
$$
|K-p-2q-r|+q\subsetneqq|K-p-q-r|,
$$
$$
|K-p-q-2r|+r\subsetneqq|K-p-q-r|.
$$

Since $\dim|K-p-q-r|=1$, we can take $D\in|K-p-q-r|$ reduced such that
$$
D\notin\left(|K-2p-q-r|+p\right)\cup\left(|K-p-2q-r|+q\right)\cup\left(|K-p-q-2r|+r\right)
$$

This last condition implies that $p,q,r\notin\Supp\left(D\right)$. Indeed, let's assume by contradiction that $p\leq D$. Then, since $D\sim K-p-q-r$, we have $0\leq D-p\sim K-2p-q-r$, i.e., $D-p\in|K-2p-q-r|$, but this implies $D=\left( D-p\right)+p\in|K-2p-q-r|+p$, which is a contradiction. Similarly, it is proven that $q,r\notin\Supp\left(D\right)$. Therefore $D+p+q+r$ is reduced.

Now let $H=\overline{D+p+q+r}$. By the Geometric Riemann-Roch Theorem
\begin{equation*}
	\begin{split}
		\dim\left(H\right)&=\deg\left(D+p+q+r\right)-\ell\left(D+p+q+r\right)\\
		&=6-4=2,
	\end{split}
\end{equation*}
i.e., $H$ is a plane in $\mathbb{P}^3$.

Let's write down $D=p_1+p_2+p_3$, so that $H\cap C=\{p_1,p_2,p_3,p,q,r\}$.

There must exist a linearly independent subset $\{q_1,q_2,q_3\}$ of $H\cap C$, i.e.,
$$
\overline{q_1+q_2+q_3}=H.
$$

Then $\mathcal{G}:\mathcal{W}_3\to\mathbb{G}\left(2,3\right)$ is defined on $q_1+q_2+q_3$ and $\mathcal{G}\left (q_1+q_2+q_3\right)=\overline{q_1+q_2+q_3}=H$, so $q_1+q_2+q_3\in\mathcal{G}^{-1}\left(H\right)$. Furthermore $\mathcal{G}$ is not defined in $p+q+r$, since $\dim\left(\overline{p+q+r}\right)=1$, i.e., $\overline{p +q+r}$ is a line. Therefore $\left(H\cdot C\right)=D+p+q+r$ is reduced, and we have that
$$
1\leq\#\left(\mathcal{G}^{-1}\left(H\right)\right)<\binom{6}{3}.
$$

This example proves that the reciprocal statement of Proposition \ref{g7} doesn't hold.

Note that in this example $W_3$ is singular, since the condition $\ell\left(p+q+r\right)=2$ implies that $\mathscr{O}_C\left(p+q+r\right)\in\Sing\left(W_3\right)$.

\begin{prop}\label{g8}
	Let $C\in\mathscr{M}_3$ be a non-hyperelliptic curve, so $W_2$ is smooth. Let $D\in C^{(2)}$, $H:=\mathcal{G}\left(D\right)$, $K=\left(H\cdot C\right)$. If $\#\left(\mathcal{G}^{-1}\left(H\right)\right)<\binom{4}{2}$, then $K$ is non-reduced.
\end{prop}

\begin{proof}
	We will prove the contrapositive. Let's assume that $K$ is reduced. Since $W_2$ is smooth, we have
	$$
	\mathcal{G}^{-1}\left(H\right)=\{E\in C^{(2)}:E\leq K\},
	$$
	which has exactly $\binom{4}{2}$ elements. This proves the proposition.	
\end{proof}

By \cite[Proposition 11.2.8.]{BL} if $C\in\mathscr{M}_g$ in non-hyperelliptic and $g\geq4$ then we have that $\dim\left(\Sing\left(W_{g-1}\right)\right)=g-4$, so $W_{g-1}$ is singular. Question: Is there a condition on $W_{g-1}$ for $g\geq4$ such that if the fiber $\mathcal{G}^{-1}\left(H\right)$ has lower cardinality than the generic fiber, then $K=\left(H\cdot C\right)$ is non-reduced?

\section{Multiple locus of the Gauss map}\label{lm}

This section is divided into three subsections. In the first we introduce the notion of multiple locus of a morphism of varieties, to then study the multiple locus of the Gauss map in $\mathcal{W}_n$. We focus on the non-hyperelliptic case with $2\leq n\leq g-2$, the other cases being well known. We first show that, in this case, the multiple locus of the Gauss map is a proper closed subset of $\mathcal{W}_n$. In the second subsection we study conditions under which the multiple locus of the Gauss map is or isn't empty in terms of Brill-Noether Theory. The third subsection, being the longest of this section, is dedicated to proving Theorem \ref{lm1}, which provides a description of the multiple locus of the Gauss map. For this we use several basic topological results.

\subsection{Definition and basic facts}

\begin{defn}
	Let $f:X\to Y$ be a morphism of quasi-projective varieties. Consider the set
	$$
	S_f:=\{x\in X:\text{exists }z\in X, \ z\neq x, \ f\left(z\right)=f\left(x\right)\}.
	$$
	
	We define the \textit{multiple locus} $R_f\subset X$ \textit{of} $f$ as the closure of the set $S_f$, i.e., $R_f:=\overline{S_f}$.
\end{defn}

Consider the multiple locus $R_{\mathcal{G}_n}$ of the Gauss map $\mathcal{G}_n:\mathcal{W}_n\to\mathbb{G}\left(n-1,g-1\right)$. We will write simply $S_n=S_{\mathcal{G}_n}$ and $R_n=R_{\mathcal{G}_n}$. Note that, although $\mathcal{W}_n:=\rho^{-1}\left(\left(W_n\right)_\sm\right)$ is open in $C^{(n)}$, when we consider $\mathcal{W}_n$ as the total space, then $\mathcal{W}_n$ is at the same time open and closed. Now, if $C$ is hyperelliptic, then $R_n=\mathcal{W}_n$ by Theorem \ref{g1}.(a). If $C$ is non-hyperelliptic: $R_1=\varnothing$; $R_{g-1}=\mathcal{W}_{g-1}$ by Theorem \ref{g1}.(b); if $2\leq n\leq g-2$, $R_n$ is a proper closed subset of $\mathcal{W}_n$. More precisely, we have the following. Recall the definition of the closed set $\mathcal{C}_n:=p\left(s^{-1}\left(C_{n+1}^1\right)\right)$ from the previous chapter.

\begin{prop}\label{lm11}
	Let $C\in\mathscr{M}_g$ be non-hyperelliptic and $2\leq n\leq g-2$. Then $R_n\subset\mathcal{C}_n\cap\mathcal{W}_n$.
\end{prop}

\begin{proof}
	If $S_n=\varnothing$, then $R_n=\varnothing\subset\mathcal{C}_n\cap\mathcal{W}_n$. Let's then assume that $S_n\neq\varnothing$. Let $D\in S_n$, so that there exists $E\in\mathcal{W}_n$ such that $E\neq D$ and $\overline{E}=\overline{D}$. This implies that $\{D,E\}\subset\mathcal{G}^{-1}\left(\overline{D}\right)$. We must prove that $D\in\mathcal{C}_n$. By contradiction, let's assume $D\notin\mathcal{C}_n$, then $D$ is in the open set $\mathcal{U}_n$ of Proposition \ref{g4}. As we saw in the proof of Theorem \ref{g1}.(c) we have $\{D,E\}\subset\mathcal{G}^{-1}\left(\overline{D}\right) =\{D\}$, which is a contradiction, since $E\neq D$. Therefore we must have that $D\in\mathcal{C}_n$. This proves that $S_n\subset\mathcal{C}_n\cap\mathcal{W}_n$. Taking closure in $\mathcal{W}_n$ gives $R_n\subset\mathcal{C}_n\cap\mathcal{W}_n$.
\end{proof}

Next we will study conditions so that $S_n\neq\varnothing$.

\begin{prop}\label{lm10}
	Let $C\in\mathscr{M}_g$ be non-hyperelliptic and $2\leq n\leq g-2$. Then $S_n\neq\varnothing$ if and only if $C_{n+1}^1\neq\varnothing$
\end{prop}

\begin{proof}
	Let's first assume that $S_n\neq\varnothing$. So
	$\varnothing\neq S_n\subset R_n\subset\mathcal{C}_n\cap\mathcal{W}_n\subset\mathcal{C}_n\Rightarrow\mathcal{C}_n\neq\varnothing$. Since $\mathcal{C}_n:=p\left(s^{-1}\left(C_{n+1}^1\right)\right)$, this implies that $C_{n+1}^1 \neq\varnothing$.	
	
	Conversely, suppose that $C_{n+1}^1\neq\varnothing$. Since $g-\left(n+1\right)-1=g-n\geq2$, by \cite[IV, (1.7) Lemma]{ACGH} no irreducible component of $C_{n+1}^1$ is completely contained in $C_{n+1}^2$. This implies that there exists $D\in C_{n+1}^1\setminus C_{n+1}^2$, i.e., $\deg\left(D\right)=n+1$ and $\dim| D|=1$.
	
	We have two cases: $|D|$ is base point free or $|D|$ has base points.
	
	Let's first assume that $|D|$ is base point free. Then by Bertini's Theorem, there is a reduced $E\in|D|$. By the Geometric Riemann-Roch Theorem, we have $\dim\left(\overline{E}\right)=n-1$. Now, since $E$ is reduced we can take $p_1,p_2\in\Supp\left(E\right)$ such that $p_1\neq p_2$. Since $|E|=|D|$ is base point free, then $p_1$ and $p_2$ aren't base points of $|E|$, so by the Geometric Riemann-Roch Theorem
	$$
	\dim\left(\overline{E-p_1}\right)=\dim\left(\overline{E-p_2}\right)=n-1,
	$$
	which implies $\overline{E-p_1}=\overline{E-p_2}=\overline{E}$, and $E-p_1\neq E-p_2$. Therefore $E-p_1\in S_n\neq\varnothing$.
	
	Now suppose that $|D|$ has base points and let $B$ be its base locus. Since $D-B$ is a special divisor, by Clifford's Theorem we have
	$$
	1=\dim|D-B|<\frac{\deg\left(D-B\right)}{2}\Rightarrow\deg\left(D-B\right)\geq3.
	$$
	
	By Bertini's Theorem, there exists $E\in|D-B|$ reduced, since $|D-B|$ is base point free. Note that $\ell\left(E+B\right)=2$, since $E+B\in|D|$. Then, by the Geometric Riemann-Roch Theorem $\dim\left(\overline{E+B}\right)=n-1$. Since $E$ is reduced and $\deg\left(E\right)=\deg\left(D-B\right)\geq3$, we can take $p_1,p_2\in\Supp\left(E\right)$ such that $p_1\neq p_2$. Since these $p_1,p_2$ aren't base points of $|E+B|$, by the Geometric Riemann-Roch Theorem we have
	$$
	\dim\left(\overline{E+B-p_1}\right)=\dim\left(\overline{E+B-p_2}\right)=n-1,
	$$
	which implies $\overline{E+B-p_1}=\overline{E+B-p_2}=\overline{E+B}$, and $E+B-p_1\neq E+B-p_2$. Therefore $E+B-p_1\in S_n\neq\varnothing$. This proves the proposition.
\end{proof}

\begin{prop}\label{lm12}
	Let $C\in\mathscr{M}_g$ and $n\geq0$. Then $C_{n+1}^1\neq\varnothing$ if and only if $C$ has a $\mathfrak{g}_{n+1}^1$.
\end{prop}

\begin{proof}
	
	Suppose $C_{n+1}^1\neq\varnothing$. Then given $D\in C_{n+1}^1$, we have $\deg\left(D\right)=n+1$ and $\ell\left(D\right)\geq2$. Then we can take a $\mathbb{F}$-vector subspace $V$ of $L\left(D\right)$ of dimension $2$, so that $\mathbb{P}\left(V\right)$ is a $\mathfrak{g}_{n+1}^1$.
	
	Conversely, suppose $C$ has a $\mathfrak{g}_{n+1}^1$, say $L$. Then, given $D\in L$ we have $\deg\left(D\right)=n+1$ and $\dim|D|\geq1$, i.e., $D\in C_{n+1}^ 1\neq\varnothing$. This proves the proposition.
\end{proof}

\begin{cor}\label{lm13}
	Let $C\in\mathscr{M}_g$ be non-hyperelliptic and $2\leq n\leq g-2$. We have $S_n\neq\varnothing$ if and only if $C$ has a $\mathfrak{g}_{n+1}^1$.
\end{cor}

\begin{proof}
	It follows directly from Propositions \ref{lm10} and \ref{lm12}.
\end{proof}

\subsection{An application of Brill-Noether Theory}

In this subsection we will determine when $S_n\neq\varnothing$ in terms of Brill-Noether Theory.

\begin{prop}
	Let $2\leq n\leq g-2$. We have the following:
	\begin{itemize}
		\item[(a)] If $g\leq2n$, for every non-hyperelliptic $C\in\mathscr{M}_g$ we have $S_n\neq\varnothing$.
		
		\item[(b)] If $g>2n$, there exists an open (dense) subset $\mathscr{U}_n$ of $\mathscr{M}_g$ such that: $C\in\mathscr{M}_g$ is non-hyperelliptic and $S_n=\varnothing$ if and only if $C\in\mathscr{U}_n$.
		
		\item[(c)] If $g>2n$, there exists a subvariety $\mathscr{V}_n$ of $\mathscr{M}_g$ of positive codimension and positive dimension such that: $C\in\mathscr{ M}_g$ is non-hyperelliptic and $S_n\neq\varnothing$ if and only if $C\in\mathscr{V}_n$.
	\end{itemize}
\end{prop}

\begin{proof}
	(a) Hypothesis $g\leq2n$ is equivalent to $\rho\left(g,1,n+1\right)\geq0$, where 
	$$
	\rho\left(g,r,d\right):=g-\left(r+1\right)\left(g-d+r\right).
	$$
	Then, by \cite[V. (1.1)]{ACGH} we have $\mathscr{M}_{g,n+1}^1=\mathscr{M}_g$, i.e., each $C\in\mathscr{ M}_g$ has a $\mathfrak{g}_{n+1}^1$. In particular, each $C\in\mathscr{M}_g\setminus\mathscr{H}_g$ has a $\mathfrak{g}_{n+1}^1$, where $\mathscr{H}_g=\mathscr{M}_{g,2}^1$ is the hyperelliptic locus. By Corollary \ref{lm13} for each $C\in\mathscr{M}_g\setminus\mathscr{H}_g$ we have $S_n\neq\varnothing$. This proves item (a).

	(b) Hypothesis $g>2n$ is equivalent to $\rho\left(g,1,n+1\right)<0$. Then, by \cite[V. (1.5)]{ACGH} we have that $\mathscr{M}_{g,n+1}^1$ is a proper closed subvariety of $\mathscr{M}_g$. Note that $\mathscr{H}_g=\mathscr{M}_{g,2}^1\subset\mathscr{M}_{g,n+1}^1$, so $\mathscr{M}_{g,n+1}^1\neq\varnothing$. By Corollary \ref{lm13} the open set $\mathscr{U}_n:=\mathscr{M}_g\setminus\mathscr{M}_{g,n+1}^1$ satisfies the property of item (b).
	
	(c) As in item (b), the hypothesis $g>2n$ implies that $\mathscr{M}_{g,n+1}^1\subsetneqq\mathscr{M}_g$. Moreover, we have that $\mathscr{H}_g=\mathscr{M}_{g,2}^1\subset\mathscr{M}_{g,n+1}^1$. We claim that $\mathscr{H}_g\subsetneqq\mathscr{M}_{g,n+1}^1$ (recall that $\mathscr{H}_g$ is irreducible and of dimension $2g-1$). Indeed, by \cite[Theorem 0.1]{S} each component $\mathscr{C}$ of $\mathscr{M}_{g,n+1}^1$ satisfies
	\begin{equation*}
		\begin{split}
			\dim\left(\mathscr{C}\right)\geq3g-3+\rho\left(g,1,n+1\right)&=2g-3+2n\\
			&\geq2g-3+4 \\
			&=2g+1>2g-1=\dim\left(\mathscr{H}_g\right).
		\end{split}
	\end{equation*}
	
	In particular, if $\mathscr{C}$ is an irreducible component of $\mathscr{M}_{g,n+1}^1$ that contains $\mathscr{H}_g$, by the previous dimension inequality we must have that $\mathscr{ H}_g\subsetneqq\mathscr{C}\subset\mathscr{M}_{g,n+1}^1$, which proves the claim.
	
	We define the variety $\mathscr{V}_n:=\mathscr{M}_{g,n+1}^1\setminus\mathscr{H}_g\neq\varnothing$. By Corollary \ref{lm13} this variety $\mathscr{V}_n$ satisfies the property of item (c). This completes the proof of the proposition.
\end{proof}

\subsection{Description of the multiple locus of the Gauss map}

In this subsection we will obtain a description of the multiple locus of the Gauss map in $\mathcal{W}_n$, by proving the reciprocal inclusion of Proposition \ref{lm11}.

\begin{thm}\label{lm1}
	Let $C\in\mathscr{M}_g$ be non-hyperelliptic and $2\leq n\leq g-2$. Then $R_n=\mathcal{C}_n\cap\mathcal{W}_n$.
\end{thm}

To prove this theorem we will use some notations. Given a complete linear system $L$, we define $L_\red$ as follows: If $\dim\left(L\right)=0$, we define $L_\red:=L$. Suppose $\dim(L)\geq1$. If $L$ is base point free, we define $L_\red:=\{D\in L:D\text{ reduced}\}$. If $L$ has a base locus $B$, we define $L_\red:=\{D\in L:D-B\text{ reduced}\}$. In any case $L_\red$ is an open dense subset of $L$ by Bertini's Theorem.

For $r\geq0$ and $d\geq1$ let's consider the closed variety (see \cite[IV. \textsection3]{ACGH}) 
$$
W_d^r\left(C\right):=\{|D|:\deg\left(D\right)=d, \ \dim|D|\geq r\}.
$$

Observe that: $C^{(n+1)}=\bigcup_{L\in W_{n+1}^0\left(C\right)}L$; $C_{n+k}^k=\bigcup_{L\in W_{n+k}^k\left(C\right)}L$ for each $k\geq1$, in particular $C_{n+1}^1=\bigcup_{L\in W_{n+1}^1\left(C\right)}L$. We define $C_\red^{(n+1)}:=\bigcup_{L\in W_{n+1}^0\left(C\right)}L_\red$.

In what follows $C\in\mathscr{M}_g$ is non-hyperelliptic and $2\leq n\leq g-2$. 

Now we proceed to prove Theorem \ref{lm1}. The inclusion $R_n\subset\mathcal{C}_n\cap\mathcal{W}_n$ was proved in Proposition \ref{lm11}. Thus we just need to prove $\mathcal{C}_n\cap\mathcal{W}_n\subset R_n$. To prove this the idea will be to show the inclusion $p\left(s^{-1}\left(C_{n+1}^1\right)\cap s^{-1}\left(C_\red^{ (n+1)}\right)\right)\cap\mathcal{W}_n\subset S_n$ (see Proposition \ref{lm2}), then take closures on this inclusion to deduce that $\mathcal{C}_n\cap\mathcal{W}_n\subset R_n$. To prove this last inclusion we are going to decompose $s^{-1}\left(C_{n+1}^1\right)$ into its irreducible components and prove that the operations between these sets behave well when taking closures. To do this we will use several basic topological results. In particular we will need to use the fact that the morphism $s$ is open, since for these morphisms the closure of the preimage of a set is the preimage of the closure of the set.

\begin{lema}
	The morphism $s$ is open.
\end{lema}

\begin{proof}
	First note that $s$ is a projective morphism, since $C^{(n)}\times C$ is a projective variety. This implies that $s$ is of finite type. Moreover, each fiber of $s$ is a finite set. From this it follows that $s$ is a quasi-finite morphism. Since $s$ is quasi-finite, surjective morphism between regular varieties, then $s$ is a flat morphism. Now, since $s$ is a flat morphism, of finite type and between Noetherian schemes, we obtain that $s$ is an open morphism.
\end{proof}

To begin with the proof of the Theorem \ref{lm1}, our first objective will be to prove four results on closures of the sets involved that we will then use for the decomposition of $s^{-1}\left(C_{n+1}^ 1\right)$ into its irreducible components. This is the longest and most technical proof in this article although it only depends on several lemmas based on basic facts of set theory and topology. First we will prove that $C_\red^{(n+1)}$ is dense in $C^{(n+1)}$.

\begin{lema}\label{lm9}
	We have $\overline{C_\red^{(n+1)}}=C^{(n+1)}$.
\end{lema}

\begin{proof}
First we have $C_\red^{(n+1)}\subset C^{(n+1)}$, so $\overline{C_\red^{(n+1)}}\subset\overline{C^{(n+1)}}=C^{(n+1)}$. On the other hand,
	\begin{equation*}
		\begin{split}
			\overline{C_\red^{(n+1)}}=\overline{\bigcup_{L\in W_{n+1}^0\left(C\right)}L_\red}\supset\bigcup_{L\in W_{n+1}^0\left(C\right)}\overline{L_\red}=\bigcup_{L\in W_{n+1}^0\left(C\right)}L=C^{(n+1)}.
		\end{split}
	\end{equation*}
	
	This proves the lemma.
\end{proof}

Using this lemma and the fact that $s$ is open, we obtain that $s^{-1}\left(C_\red^{(n+1)}\right)$ is dense in $C^{ (n)}\times C$.

\begin{cor}
	We have $\overline{s^{-1}\left(C_\red^{(n+1)}\right)}=C^{(n)}\times C$.
\end{cor}

\begin{proof}
	Using Lemma \ref{lm9}, the fact that $s$ is open, we obtain
	$$
	\overline{s^{-1}\left(C_\red^{(n+1)}\right)}=s^{-1}\left(\overline{C_\red^{(n+1) }}\right)=s^{-1}\left(C^{(n+1)}\right)=C^{(n)}\times C.
	$$	
\end{proof}

Now we will prove that $C_{n+1}^1\cap C_\red^{(n+1)}$ is dense in $C_{n+1}^1$.

\begin{lema}\label{lm4}
	We have $\overline{C_{n+1}^1\cap C_\red^{(n+1)}}=C_{n+1}^1$.
\end{lema}

\begin{proof}
	Since $C_{n+1}^1\cap C_\red^{(n+1)}\subset C_{n+1}^1$, then $\overline{C_{n+1}^1\cap C_\red^{(n+1)}}\subset\overline{C_{n+1}^1}=C_{n+1}^1$. On the other hand
	\begin{equation*}
		\begin{split}
			\overline{C_{n+1}^1\cap C_\red^{(n+1)}}=\overline{\bigcup_{L\in W_{n+1}^1\left(C\right)}L_\red}\supset\bigcup_{L\in W_{n+1}^1\left(C\right)}\overline{L_\red}=\bigcup_{L\in W_{n+1}^1\left(C\right)}L=C_{n+1}^1,
		\end{split}
	\end{equation*}
which proves the lemma.
\end{proof}

Using this lemma and the fact that $s$ is open, we will deduce that $s^{-1}\left(C_{n+1}^1\right)\cap s^{-1}\left(C_\red ^{(n+1)}\right)$ is dense in $s^{-1}\left(C_{n+1}^1\right)$.

\begin{cor}\label{lm5}
	We have $\overline{s^{-1}\left(C_{n+1}^1\right)\cap s^{-1}\left(C_\red^{(n+1)}\right)}=s^{-1}\left(C_{n+1}^1\right)$.
\end{cor}

\begin{proof}
By Lemma \ref{lm4} and since $s$ is an open morphism, we have
	\begin{equation*}
		\begin{split}
			\overline{s^{-1}\left(C_{n+1}^1\cap C_\red^{(n+1)}\right)}=s^{-1}\left(\overline{ C_{n+1}\cap C_\red^{(n+1)}}\right)=s^{-1}\left(C_{n+1}^1\right),
		\end{split}
	\end{equation*}
\end{proof}

This achieve our first objective. Now, our second objective to prove the Theorem \ref{lm1}, will be to show, through four results, that for each irreducible component $X$ of $s^{-1}\left(C_{n+1}^ 1\right)$ such that $p\left(X\right)\cap\mathcal{W}_n\neq\varnothing$, we have that $p\left(X\cap s^{-1}\left( C_\red^{(n+1)}\right)\right)\cap\mathcal{W}_n$ is dense in $p\left(X\right)$ (see Lemma \ref{lm3}). To do this we will first prove, using the previous results, that for each of these irreducible components $X$ we have that $X\cap s^{-1}\left(C_\red^{(n+1)}\right) $ is dense in $X$, and, thus, irreducible.

\begin{lema}\label{lm6}
	Let $X_1,\dots,X_m$ be the irreducible components of $s^{-1}\left(C_{n+1}^1\right)$. Then
	$$
	\overline{X_j\cap s^{-1}\left(C_\red^{(n+1)}\right)}=X_j
	$$
	for each $j$. Moreover $X_j\cap s^{-1}\left(C_\red^{(n+1)}\right)$ is irreducible for every $j$.
\end{lema}

\begin{proof}
	By Lemma \ref{lm5}
	\begin{equation*}
		\begin{split}
			X_1\cup\dots\cup X_m=s^{-1}\left(C_{n+1}^1\right)&=\overline{s^{-1}\left(C_{n+1}^1\right)\cap s^{-1}\left(C_\red^{(n+1)}\right)}\\
			&=\overline{X_1\cap s^{-1}\left(C_\red^{(n+1)}\right)}\cup\dots\cup\overline{X_m\cap s^{-1}\left(C_\red^{(n+1)}\right)}.
		\end{split}
	\end{equation*}
	
	Then, given $j\in\{1,\dots,m\}$, for some $i\in\{1,\dots,m\}$ we have
	$$
	X_j\subset\overline{X_i\cap s^{-1}\left(C_\red^{(n+1)}\right)}\subset\overline{X_i}\cap\overline{s^{-1}\left(C_\red^{(n+1)}\right)}=X_i,
	$$
	therefore $i=j$, since the $X_1,\dots,X_m$ are the irreducible components of $s^{-1}\left(C_{n+1}^1\right)$.
	
	Thus $X_j=\overline{X_j\cap s^{-1}\left(C_\red^{(n+1)}\right)}$. 
	
	Finally each $X_j\cap s^{-1}\left(C_\red^{(n+1)}\right)$ is irreducible, since each $X_j$ is irreducible.
\end{proof}

Next, using the previous lemma, we will prove that for each component $X$ of $s^{-1}\left(C_{n+1}^1\right)$ the set $p\left(X\cap s^{-1}\left(C_\red^{(n+1)}\right)\right)$ is dense in $p\left(X\right)$, and, thus, irreducible.

\begin{lema}\label{lm8}
	Let $X_1,\dots,X_m$ be the irreducible components of $s^{-1}\left(C_{n+1}^1\right)$. Then
	$$
	\overline{p\left(X_i\cap s^{-1}\left(C_\red^{(n+1)}\right)\right)}=p\left(X_j\right)
	$$
	for each $j$. Moreover $p\left(X_j\cap s^{-1}\left(C_\red^{(n+1)}\right)\right)$ is irreducible for every $j$.
\end{lema}

\begin{proof}
	By Lemma \ref{lm6}, for each $j\in\{1,\dots,m\}$ we have
	\begin{equation*}
		\begin{split}
			\overline{p\left(X_j\cap s^{-1}\left(C_\red^{(n+1)}\right)\right)}=\overline{p\left(\overline{X_j\cap s^{-1}\left(C_\red^{(n+1)}\right)}\right)}=\overline{p\left(X_j\right)}=p\left(X_i\right).
		\end{split},
	\end{equation*}
	and $p\left(X_j\cap s^{-1}\left(C_\red^{(n+1)}\right)\right)$ is irreducible, because $p\left(X_j\right)$ is irreducible.
\end{proof}

To achieve our second objective we will use the following lemma.

\begin{lema}\label{lm7}
	Let $X_1,\dots,X_m$ be the irreducible components of $s^{-1}\left(C_{n+1}^1\right)$. Then, we have that $p\left(X_j\right)\cap\mathcal{W}_n\neq\varnothing$ if and only if $p\left(X_j\cap s^{-1}\left(C_\red^{(n+1)}\right)\right)\cap\mathcal{W}_n\neq\varnothing$.
\end{lema}

\begin{proof}
	To prove the non-trivial implication assume $p\left(X_j\right)\cap\mathcal{W}_n\neq\varnothing$. Given $D\in p\left(X_j\right)\cap\mathcal{W}_n$, there exists $q\in C$ such that $(D,q)\in X_j$ and $D=p\left( D,q\right)$, so $\left(D,q\right)\in X_j\cap\left(\mathcal{W}_n\times C\right)\neq\varnothing$. Now, since $X_j\cap\left(\mathcal{W}_n\times C\right)$ is open (dense) on $X_j$ (irreducible), and $X_j\cap s^{-1}\left (C_\red^{(n+1)}\right)$ is dense in $X_j$ by Lemma \ref{lm6}, then $X_j\cap s^{-1}\left(C_\red ^{(n+1)}\right)\cap\left(\mathcal{W}_n\times C\right)$ must be dense in $X_j$, so this intersection must be non-empty. Then, given $\left(D,p\right)\in X_j\cap s^{-1}\left(C_\red^{(n+1)}\right)\cap\left(\mathcal{W }_n\times C\right)$, we have
	$$
	D=p\left(D,q\right)\in p\left(X_j\cap s^{-1}\left(C_\red^{(n+1)}\right)\right)\cap\mathcal{W}_n\neq\varnothing.
	$$
	
	This proves the lemma.
\end{proof}

From this lemma we deduce the following result, with which we achieve our second objective.

\begin{cor}\label{lm3}
	Let $X_1,\dots,X_m$ be the irreducible components of $s^{-1}\left(C_{n+1}^1\right)$. Then we have that if $p\left(X_j\right)\cap\mathcal{W}_n\neq\varnothing$, then $p\left(X_j\cap s^{-1}\left(C_\red^ {(n+1)}\right)\right)\cap\mathcal{W}_n$ is dense in $p\left(X_j\right)$.
\end{cor}

\begin{proof}
	Assume $p\left(X_j\right)\cap\mathcal{W}_n\neq\varnothing$. Then, by Lemma \ref{lm7}, we have
	$$
	p\left(X_j\cap s^{-1}\left(C_\red^{(n+1)}\right)\right)\cap\mathcal{W}_n\neq\varnothing.
	$$
	
Then $p\left(X_j\cap s^{-1}\left(C_\red^{(n+1)}\right)\right)\cap\mathcal{W}_n$ is open (dense) in $p\left(X_j\cap s^{-1}\left(C_\red^{(n+1)}\right)\right)$, because this last set is irreducible by Lemma \ref{lm8}. Since $p\left(X_j\cap s^{-1}\left(C_\red^{(n+1)}\right)\right)$ is dense in $p\left(X_j\right)$ also by Lemma \ref{lm8}, then $p\left(X_j\cap s^{-1}\left(C_\red^{(n+1)}\right)\right)\cap\mathcal{ W}_n$ must be dense in $p\left(X_j\right)$.
\end{proof}

This achieve our second objective. Now we proceed to prove the inclusion on which we will then take closures and apply the previous results.

\begin{prop}\label{lm2}
	We have that $p\left(s^{-1}\left(C_{n+1}^1\right)\cap s^{-1}\left(C_\red^{(n+1)}\right)\right)\cap\mathcal{W}_n\subset S_n$.
\end{prop}

\begin{proof}
	If $S_n=\varnothing$, we trivially obtain the inclusion. Suppose then that $S_n\neq\varnothing$. Then, by Proposition \ref{lm11} we have $\mathcal{C}_n\cap\mathcal{W}_n\neq\varnothing$. By Lemma \ref{lm7} this implies that
	$$
	p\left(s^{-1}\left(C_{n+1}^1\right)\cap s^{-1}\left(C_\red^{(n+1)}\right)\right)\cap\mathcal{W}_n\neq\varnothing.
	$$
	
	Then, given $D\in p\left(s^{-1}\left(C_{n+1}^1\right)\cap s^{-1}\left(C_\red^{(n+ 1)}\right)\right)\cap\mathcal{W}_n$, we have $D=p\left(D,q\right)$ for some $q\in C$ such that $\left(D ,q\right)\in s^{-1}\left(C_{n+1}^1\cap C_\red^{(n+1)}\right)$, so $D+q=s \left(D,q\right)\in C_{n+1}^1\cap C_\red^{(n+1)}$. Now, since $D\in\mathcal{W}_n$, we must have that $\dim|D+q|=1$. By the Geometric Riemann-Roch Theorem this implies that $\dim\left(\overline{D+q}\right)=n-1$, so $\overline{D+q}=\overline{D}$.
	
	We have two cases: $|D+q|$ is base point free or has base points.
	
	Let's first assume that $|D+q|$ is base point free. Then $D+q$ is reduced, since $D+q\in C_\red^{(n+1)}$. Then, taking $q_1\in\Supp\left(D\right)$ we have that $q_1$ isn't a base point of $|D+q|$, so
	$$
	\overline{D+q-q_1}=\overline{D+q}=\overline{D},
	$$
	and $D+q-q_1\neq D$, since $q_1\neq q$, which proves $D\in S_n$.
	
	Now suppose that $|D+q|$ has base points. Let $B$ be the base locus of $|D+q|$. Then
	$$
	|D+q|=|D+q-B|+B,
	$$
	where $|D+q-B|$ is base point free. Note that $q$ isn't a base point of $|D+q|$, i.e., $q\notin\Supp\left(B\right)$. Since $D+q-B$ is special divisor, by Clifford's Theorem we have
	$$
	1=\dim|D+q-B|<\frac{\deg\left(D+q-B\right)}{2}\Rightarrow\deg\left(D+q-B\right)\geq3.
	$$
	
	Since $D+q-B$ is reduced, we can take $q_1\in\Supp\left(D-B\right)$ such that $q_1\neq q$. Since $q_1$ is not a base point of $|D+q|$ we have
	$$
	\overline{D+q-q_1}=\overline{D+q}=\overline{D},
	$$
	and $D+q-q_1\neq D$, since $q_1\neq q$, which proves $D\in S_n$. This proves the proposition.
\end{proof}

Now, using this proposition, we will prove Theorem \ref{lm1}.

\begin{proof}[Proof of Theorem \ref{lm1}]
	If $C_{n+1}^1=\varnothing$ is clear that $R_n=\varnothing=\mathcal{C}_n\cap\mathcal{W}_n$. 
	
	Suppose then that $C_{n+1}^1\neq\varnothing$. This implies that $S_n\neq\varnothing$ by Proposition \ref{lm10}. Then we have that $p\left(s^{-1}\left(C_{n+1}^1\right)\cap s^{-1}\left(C_\red^{(n+1)}\right )\right)\cap\mathcal{W}_n\neq\varnothing$ by Propositions \ref{lm11} and \ref{lm7}.

	The inclusion $R_n\subset\mathcal{C}_n\cap\mathcal{W}_n$ was proved in Proposition \ref{lm11}. Then we just need to prove $\mathcal{C}_n\cap\mathcal{W}_n\subset R_n$. As we said before, the main idea will be to take closures in the inclusion
	$$
	p\left(s^{-1}\left(C_{n+1}^1\right)\cap s^{-1}\left(C_\red^{(n+1)}\right)\right)\cap\mathcal{W}_n\subset S_n
	$$
	given by Proposition \ref{lm2}.

	Let $X_1,\dots,X_m$ be the irreducible components of $s^{-1}\left(C_{n+1}^1\right)$. Since $p$ is a proper and surjective morphism, then each irreducible component of $\mathcal{C}_n$ is some of the $p\left(X_1\right),\dots,p\left(X_m\right)$. Discarding the possible sets that are contained in the union of the others and rearranging if necessary, without loss of generality we can assume that $p\left(X_1\right),\dots,p\left(X_k\right)$ are the irreducible components of $\mathcal{C}_n$, without redundant sets, where $k\leq m$.
	
	Since $\mathcal{C}_n\cap\mathcal{W}_n\supset S_n\neq\varnothing$, there exists $j\in\{1,\dots,k\}$ such that $p\left(X_j\right )\cap\mathcal{W}_n\neq\varnothing$. Then, rearranging if necessary, without loss of generality we can assume that there exists $d\leq k$ such that $p\left(X_j\right)\cap\mathcal{W}_n\neq\varnothing$ for all $j\in\{1,\dots,d\}$, and $p\left(X_j\right)\cap\mathcal{W}_n=\varnothing$ for each $j\in\{d+1,\dots, k\}$ in the case that $d<k$. By Proposition \ref{lm2} and Lemma \ref{lm7} this implies that
	\begin{equation*}
		\begin{split}
			S_n&\supset p\left(s^{-1}\left(C_{n+1}\right)\cap s^{-1}\left(C_\red^{(n+1)}\right)\right)\cap\mathcal{W}_n\\
				&\supset\left(p\left(X_1\cap s^{-1}\left(C_\red^{(n+1)}\right)\right)\cap\mathcal{W}_n\right)\cup\dots\cup\left(p\left(X_d\cap s^{-1}\left(C_\red^{(n+1)}\right)\right)\cap\mathcal{W}_n\right).\\
		\end{split}
	\end{equation*}
	
	By Corollary \ref{lm3}, taking closures we obtain
	\begin{equation*}
		\begin{split}
			\overline{S_n}&\supset\overline{p\left(X_1\cap s^{-1}\left(C_\red^{(n+1)}\right)\right)\cap\mathcal{W}_n}\cup\dots\cup\overline{p\left(X_d\cap s^{-1}\left(C_\red^{(n+1)}\right)\right)\cap\mathcal{W}_n}\\
			&=p\left(X_1\right)\cup\dots\cup p\left(X_d\right).
		\end{split}
	\end{equation*}	
	
	So 
	\begin{equation*}
		\begin{split}
			R_n=\overline{S_n}^{\mathcal{W}_n}=\overline{S_n}\cap\mathcal{W}_n&\supset\left(p\left(X_1\right)\cap\mathcal{W}_n\right)\cup\dots\cup\left(p\left(X_d\right)\cap\mathcal{W}_n\right)\\
			&=\left(p\left(X_1\right)\cap\mathcal{W}_n\right)\cup\dots\cup\left(p\left(X_k\right)\cap\mathcal{W}_n\right)=\mathcal{C}_n\cap\mathcal{W}_n,
		\end{split}
	\end{equation*}
	which proves the theorem.
\end{proof}

Using this theorem we can obtain a more explicit description of the multiple locus of $\mathcal{G}$.
\begin{cor}
	Let $C\in\mathscr{M}_g$ be non-hyperelliptic and $2\leq n\leq g-2$. If $S_n\neq\varnothing$, then
	$$
	R_n=\{D\in\mathcal{W}_n:\text{exists }q\in C\text{ such that }\dim|D+q|=1\}.
	$$
\end{cor}

Here the following question naturally arises: What is $R_n\setminus S_n$?

Note that this corollary generalizes the result of Smith and Tapia-Recillas \cite[p. 232, Lemma]{STR1} for arbitrary dimension $n$ and $W_n$ not necessarily smooth.

\section{Larger intersection loci}\label{lmi}

This section is divided into three subsections. In the first we introduce the notion of larger  intersection loci and we prove some basic facts about them, for example, that they are closed subsets of $\mathcal{W}_n$. We also prove some results on linear systems. The notion of larger intersection loci will be used in the next section to study the Generalized Torelli Problem in the non-hyperelliptic case. In the second subsection we study numerical conditions such that the larger intersection loci are not empty or are empty, in terms of Brill-Noether Theory. In the third subsection we will find bounds for the dimensions of the larger intersection loci.

\subsection{Definition and basic facts}

In this subsection $C\in\mathscr{M}_g$ is a non-hyperelliptic curve and $2\leq n\leq g-2$. Now, for each $k\geq0$ we define the $(n+k)$-\textit{intersection locus} as
$$
R_{n,k}:=\{D\in\mathcal{W}_n:\deg\left(\overline{D}\cdot C\right)\geq n+k\}.
$$

Note that $R_{n,0}=\mathcal{W}_n$ and $R_{n,k+1}\subset R_{n,k}$ for all $k\geq0$.

\begin{prop}
	We have $R_{n,k}=\varnothing$ for all $k\geq n$.
\end{prop}

\begin{proof}
	Suppose $R_{n,k}\neq\varnothing$ and let $D\in R_{n,k}$. We will prove $k\leq n-1$. Since $\deg\left(\overline{D}\cdot C\right)\geq n+k$, we can take $E\in C^{(n+k)}$ such that $D\leq E\leq \left(\overline{D}\cdot C\right)$. Note that $\overline{E}=\overline{D}$. Now, by the Geometric Riemann-Roch Theorem \cite[p. 12]{ACGH}
	\begin{equation*}
		\begin{split}
			\ell\left(E\right)&=\deg\left(E\right)-\dim\left(\overline{E}\right)\\
			&=\left(n+k\right)-\left(n-1\right)=k+1,
		\end{split}
	\end{equation*}
	i.e., $\dim|E|=k$. Since $E$ is a special divisor, by Clifford's Theorem we have
	$$
	k=\dim|E|<\frac{\deg\left(E\right)}{2}=\frac{n+k}{2}\Rightarrow k\leq n-1.
	$$
	
	By contrapositive, if $k\geq n$, then $R_{n,k}=\varnothing$. This proves the proposition.
\end{proof}

Now we will prove some results about linear systems.

\begin{prop}
	Let $1\leq k\leq n-1$. Then $C$ has a complete $\mathfrak{g}_{n+k}^k$ if and only if $C$ has a $\mathfrak{g}_{n+k}^k$.	
\end{prop}

\begin{proof}
	To prove the non-trivial implication suppose that $C$ has a $\mathfrak{g}_{n+k}^k$, say $L$. Let $E\in L$. Then $|E|$ is a complete $\mathfrak{g}_{n+k}^l$ for some $l\geq k$. If $l=k$ there is nothing to do. Let's assume then that $l>k$. By the Geometric Riemann-Roch Theorem
	\begin{equation*}
		\begin{split}
			\dim\left(\overline{E}\right)&=\deg\left(E\right)-\ell\left(E\right)\\
			&=\left(n+k\right)-\left(l+1\right)=n+k-l-1.
		\end{split}
	\end{equation*}
	
	We can take $P_1,\dots,P_{l-k}\in C$ general points such that
	\begin{equation*}
		\begin{split}
			\dim\left(\overline{E+P_1+\dots+P_{l-k}}\right)&=n+k-l-1+\left(l-k\right)\\
			&=n-1.
		\end{split}
	\end{equation*}
	
	Let's now take $D\in C^{(n)}$ such that $D\leq E+P_1+\dots+P_{l-k}$, and $\overline{D}=\overline{E+P_1+\dots+P_ {l-k}}$. We can also take $F\in C^{(n+k)}$ with $D\leq F\leq E+P_1+\dots+P_{l-k}$. Then we have
	$$
	\overline{F}=\overline{E+P_1+\dots+P_{l-k}}.
	$$
	
	Then, by the Geometric Riemann-Roch Theorem
	\begin{equation*}
		\begin{split}
			\ell\left(F\right)&=\deg\left(F\right)-\dim\left(\overline{F}\right)\\
			&=\left(n+k\right)-\left(n-1\right)=k+1,
		\end{split}
	\end{equation*}
	so $|F|$ is a complete $\mathfrak{g}_{n+k}^k$.
\end{proof}

Now we will study when $R_{n,k}\neq\varnothing$ in terms of existence of linear systems on $C$.

\begin{prop}\label{lmi1}
	For $1\leq k\leq n-1$ we have $R_{n,k}\neq\varnothing$ if and only if $C$ has a (complete) $\mathfrak{g}_{n+k}^k$.
\end{prop}

\begin{proof}
	Assume that $R_{n,k}\neq\varnothing$. Let $D\in R_{n,k}$, so that $\deg\left(\overline{D}\cdot C\right)\geq n+k$. Then, we can take $E\in C^{(n+k)}$ such that $D\leq E\leq\left(\overline{D}\cdot C\right)$. Note that $\overline{E}=\overline{D}$. By the Geometric Riemann-Roch Theorem
	\begin{equation*}
		\begin{split}
			\ell\left(E\right)&=\deg\left(E\right)-\dim\left(\overline{E}\right)\\
			&=\left(n+k\right)-\left(n-1\right)=k+1,
		\end{split}
	\end{equation*}
	so $|E|$ is a complete $\mathfrak{g}_{n+k}^k$. This proves the first implication.
	
	Conversely, let's assume that $C$ has a complete $\mathfrak{g}_{n+k}^k$, say $|E|$. By the Geometric Riemann-Roch Theorem
	\begin{equation*}
		\begin{split}
			\dim\left(\overline{E}\right)&=\deg\left(E\right)-\ell\left(E\right)\\
			&=\left(n+k\right)-\left(k+1\right)=n-1.
		\end{split}
	\end{equation*}
	
	Then, we can take $D\in C^{(n)}$ such that $D\leq E$ and $\overline{D}=\overline{E}$. Moreover $\left(\overline{D}\cdot C\right)=\left(\overline{E}\cdot C\right)\geq E$, so $\deg\left(\overline{D}\cdot C\right)\geq n+k$, and therefore $D\in R_{n,k}\neq\varnothing$. This proves the proposition.
\end{proof}

From this proposition is deduced the following corollary.

\begin{cor}
	For $1\leq k\leq n-1$ if $C$ has a (complete) $\mathfrak{g}_{n+k}^k$, then $C$ has a (complete) $\mathfrak{g}_ {n+k-1}^{k-1}$.
\end{cor}

\begin{proof}
	If $C$ has a complete $\mathfrak{g}_{n+k}^k$, then $R_{n,k}\neq\varnothing$ by Proposition \ref{lmi1}. Then, $R_{n,k-1}\neq\varnothing$, since $R_{n,k-1}\supset R_{n,k}$. Then $C$ has a complete $\mathfrak{g}_{n+k-1}^{k-1}$ by Proposition \ref{lmi1}.
\end{proof}

\begin{prop}\label{lmi3}
	For $1\leq k\leq n-1$ we have that $C$ has a $\mathfrak{g}_{n+k}^k$ if and only if $C_{n+k}^k\neq \varnothing$.
\end{prop}

\begin{proof}
	Suppose $C$ has a $\mathfrak{g}_{n+k}^k$, say $L$. Then, given $E\in L$, we have $\deg\left(E\right)=n+k$ and $\dim|E|\geq k$, i.e., $E\in C_{n+k }^k\neq\varnothing$.
	
	Conversely, suppose that $C_{n+k}^k\neq\varnothing$. Given $E\in C_{n+k}^k$ we have $\deg\left(E\right)=n+k$ and $\dim|E|\geq k$, i.e., $\ell\left (E\right)\geq k+1$. Then, we can take a $\mathbb{F}$-vector subspace $V$ of $L\left(E\right)$ of dimension $k+1$. Therefore $\mathbb{P}\left(V\right)$ is a $\mathfrak{g}_{n+k}^k$.	This proves the proposition.
\end{proof}

Recall the closed subsets $\mathcal{C}_{n,k}:=p\left(s^{-1}\left(C_{n+k}^k\right)\right)$ of $C^{(n)}$ defined in Subsection \ref{ck}. Using these sets we will prove that the $R_{n,k}$ are closed subsets of $\mathcal{W}_n$.

\begin{prop}\label{lmi2}
	For every $1\leq k\leq n-1$ we have that $R_{n,k}$ is closed in $\mathcal{W}_n$. More precisely $R_{n,k}=\mathcal{C}_{n,k}\cap\mathcal{W}_n$.
\end{prop}

\begin{proof}
	If $R_{n,k}=\varnothing$, then $C$ doesn't have a $\mathfrak{g}_{n+k}^k$ by Proposition \ref{lmi1}. This implies that $C_{n+k}^k=\varnothing$ by Proposition \ref{lmi3}, so $\mathcal{C}_{n,k}\cap\mathcal{W}_n=p\left(s^{ -1}\left(\varnothing\right)\right)\cap\mathcal{W}_n=\varnothing=R_{n,k}$, which is closed.
	
	Suppose then $R_{n,k}\neq\varnothing$. We will prove that $R_{n,k}=\mathcal{C}_{n,k}\cap\mathcal{W}_n$. Indeed, let $D\in R_{n,k}$, so that $\deg\left(\overline{D}\cdot C\right)\geq n+k$. Then we can take $E\in C^{(n+k)}$ such that $D\leq E\leq\left(\overline{D}\cdot C\right)$. Note that $\overline{E}=\overline{D}$. By the Geometric Riemann-Roch Theorem \cite[p. 12]{ACGH}
	\begin{equation*}
		\begin{split}
			\ell\left(E\right)&=\deg\left(E\right)-\dim\left(\overline{E}\right)=\left(n+k\right)-\left(n-1\right)=k+1,
		\end{split}
	\end{equation*}
	so $s\left(D,E-D\right)=E\in C_{n+k}^k$, which implies $D=p\left(D,E-D\right)\in p\left(s^{-1 }\left(C_{n+k}^k\right)\right)=:\mathcal{C}_{n,k}$. Thus $D\in\mathcal{C}_{n,k}\cap\mathcal{W}_n$.
	
	Conversely, let $D\in\mathcal{C}_{n,k}\cap\mathcal{W}_n$. Then $D=p\left(D,F\right)$ with $\left(D,F\right)\in s^{-1}\left(C_{n+k}^k\right)$. Then we have $D+F=s\left(D,F\right)\in C_{n+k}^k$, i.e., $\dim|D+F|\geq k$. Now, by the Geometric Riemann-Roch Theorem
	\begin{equation*}
		\begin{split}
			n-1=\dim\left(\overline{D}\right)\leq\dim\left(\overline{D+F}\right)&=\deg\left(D+F\right)-\ell\left(D+F\right)\\
			&\leq\left(n+k\right)-\left(k+1\right)=n-1,
		\end{split}
	\end{equation*}
	so $\overline{D+F}=\overline{D}$. This implies that $\left(\overline{D}\cdot C\right)=\left(\overline{D+F}\cdot C\right)\geq D+F$, so $\deg\left( \overline{D}\cdot C\right)\geq n+k$. This shows that $D\in R_{n,k}$. Therefore $R_{n,k}=\mathcal{C}_{n,k}\cap\mathcal{W}_n$, which proves the proposition.	
\end{proof}

\begin{cor}
	We have $R_{n,1}=R_n$.
\end{cor}

\begin{proof}
	By Theorem \ref{lm1} and Proposition \ref{lmi2} we have $R_n=\mathcal{C}_n\cap\mathcal{W}_n=R_{n,1}$.
\end{proof}

By Proposition \ref{lmi2} we have a stratification
$$
\mathcal{W}_n=R_{n,0}\supset R_n=R_{n,1}\supset R_{n,2}\supset\dots\supset R_{n,n-1}.
$$

\subsection{An application of Brill-Noether Theory}

In this subsection we will study conditions under which $R_{n,k}\neq\varnothing$ in terms of Brill-Noether Theory. For integers $r,d$, with $r\geq0$ and $d\geq1$, we denote $G_d^r\left(C\right):=\{L:L\text{ is a } \mathfrak{ g}_d^r\text{ in }C\}$ which has the structure of a variety (see \cite[IV. \textsection3]{ACGH}).

\begin{prop}\label{lmi6}
	Let $1\leq k\leq n-1\leq g-3$. Suppose $g\leq\frac{(k+1)n}{k}$. Then for every non-hyperelliptic curve $C\in\mathscr{M}_g$ we have $R_{n,k}\neq\varnothing$. If also $k\geq3$, then there exists an open (dense) subset $\mathscr{U}_{n,k}$ of $\mathscr{M}_g$ such that: $C\in\mathscr{M}_g$ is non-hyperelliptic and the general element of $G_{n+k}^k\left(C\right)$ is base point free if and only if $C\in\mathscr{U}_{n,k}$.
	
\end{prop}

\begin{proof}
	The hypothesis $g\leq\frac{(k+1)n}{k}$ is equivalent to $\rho\left(g,k,n+k\right)\geq0$, so by \cite[V. (1.1)]{ACGH} for all $C\in\mathscr{M}_g$ we have $G_{n+k}^k\left(C\right)\neq\varnothing$. In particular for all non-hyperelliptic $C\in\mathscr{M}_g$ we have $G_{n+k}^k\left(C\right)\neq\varnothing$. Then, by Proposition \ref{lmi1}, for all non-hyperelliptic $C\in\mathscr{M}_g$ we have $R_{n,k}\neq\varnothing$.
	
	Now if in addition $k\geq3$, by \cite[Theorem 1]{EH} there exists an open $\mathscr{U}^1_{n,k}$ of $\mathscr{M}_g$ such that for $C\in\mathscr{M }_g$  the general element of $G_{n+k}^k\left(C\right)$ is base point free if and only if $C\in\mathscr{U}^1_{n,k}$. Then, the open set $\mathscr{U}_{n,k}:=\mathscr{U}^1_{n,k}\cap\left(\mathscr{M}_g\setminus\mathscr{H}_g\right)$ satisfies the requirement of the proposition.
\end{proof}

\begin{prop}
	Let $1\leq k\leq n-1\leq g-3$. Suppose $g>\frac{\left(k+1\right)n}{k}$. Then there exists an open (dense) subset $\mathscr{S}_{n,k}$ of $\mathscr{M}_g$ such that: $C\in\mathscr{M}_g$ is non-hyperelliptic and $R_{n,k}=\varnothing$ if and only if $C\in\mathscr{S}_{n,k}$.
\end{prop}

\begin{proof}
	The hypothesis $g>\frac{(k+1)n}{k}$ is equivalent to $\rho\left(g,k,n+k\right)<0$, so by \cite[V. (1.5)]{ACGH} we have $\mathscr{M}_{g,n+k}^k\subsetneqq\mathscr{M}_g$. Consider the open set $\mathscr{S}_{n,k}:=\mathscr{M}_g\setminus\mathscr{M}_{g,n+k}^k$. Now note that if $C$ doesn't have a $\mathfrak{g}_{n+k}^k$, then $C$ cannot be hyperelliptic by Proposition \ref{a1}. Thus $\mathscr{S}_{n,k}\subset\left(\mathscr{M}_g\setminus\mathscr{H}_g\right)$. By Proposition \ref{lmi1} the open set $\mathscr{S}_{n,k}$ satisfies the required property.
\end{proof}

\begin{prop}\label{lmi7}
	Let $1\leq k\leq n-1\leq g-3$. Suppose that $g>\frac{(k+1)n}{k}$ and that $n\geq\frac{\left(k-1\right)g+3}{k+1}$. Then there exists a subvariety $\mathscr{V}_{n,k}$ of $\mathscr{M}_g$ of positive codimension and positive dimension such that: $C\in\mathscr{M}_g$ is non-hyperelliptic and $R_{n,k}\neq \varnothing$ if and only if $C\in\mathscr{V}_{n,k}$.
\end{prop}

\begin{proof}
	The hypothesis $g>\frac{(k+1)n}{k}$ is equivalent to $\rho\left(g,k,n+k\right)<0$. Then, by \cite[V. (1.5)]{ACGH} we have $\mathscr{M}_{g,n+k}^k\subsetneqq\mathscr{M}_g$. Now, if $C$ doesn't have a $\mathfrak{g}_{n+k}^k$, then $C$ cannot be hyperelliptic by Proposition \ref{a1}. This implies $\left(\mathscr{M}_g\setminus\mathscr{M}_{g,n+k}^k\right)\subset\left(\mathscr{M}_g\setminus\mathscr{H}_g\right)$, i.e., $\mathscr{H}_g\subset\mathscr{M}_{g,n+k}^k$. Recall that $\mathscr{H}_g$ is irreducible and of dimension $2g-1$. We claim that $\mathscr{H}_g\subsetneqq\mathscr{M}_{g,n+k}^k$.
	
	By Theorem \cite[Theorem 0.1]{S} for each irreducible component $\mathscr{C}$ of $\mathscr{M}_{g,n+k}^k$ we have
	\begin{equation*}
		\begin{split}
			\dim\left(\mathscr{C}\right)\geq3g-3+\rho\left(g,k,n+k\right)&=\left(3-k\right)g-3+\left (k+1\right)n\\
			&\geq\left(3-k\right)g-3+\left(k-1\right)g+3\\
			&=2g>2g-1=\dim\left(\mathscr{H}_g\right).
		\end{split}
	\end{equation*}
	
	In particular, if $\mathscr{C}$ is an irreducible component of $\mathscr{M}_{g,n+k}^k$ that contains $\mathscr{H}_g$, by the previous dimension inequality we must have that $\mathscr{ H}_g\subsetneqq\mathscr{C}\subset\mathscr{M}_{g,n+k}^k$, which proves the claim.
	
	Define $\mathscr{V}_{n,k}:=\mathscr{M}_{g,n+k}^k\setminus\mathscr{H}_g$. By Proposition \ref{lmi1}, this variety $\mathscr{V}$ satisfies the requirements.
\end{proof}

\subsection{Bounds for the dimension}

In this subsection we will obtain some bounds for the dimension of the larger intersection loci and their images via the Gauss map. For each $0\leq k\leq n-1$ we define $B_{n,k}:=\mathcal{G}\left(R_{n,k}\right)\subset\mathbb{G}\left(n-1,g-1\right)$ . Note that $B_{n,k}$ is not necessarily a variety, but just a constructible set. Now for integers $r,d$, with $r\geq0$ and $d\geq1$, we define
$$
\mathcal{L}_d^r\left(C\right):=\{L:L\text{ is a complete }\mathfrak{g}_d^r\text{ in }C\}.
$$

\begin{prop}\label{lmi8}
	Let $1\leq k\leq n-1$. Assume that $C$ has a complete $\mathfrak{g}_{n+k}^k$. Then the map
	$$
	\beta:\bigcup_{L\in\mathcal{L}_{n+k}^k\left(C\right)}L\to B_{n,k}, \ E\mapsto\overline{E}
	$$
	is regular, surjective and has finite fibers.
\end{prop}

\begin{proof}
	Let's first check that $\beta$ is well-defined. Let $E\in L\in\mathcal{L}_{n+k}^k\left(C\right)$, i.e., $\ell\left(E\right)=k+1$. We want to check that $\overline{E}\in B_{n,k}$. By the Geometric Riemann-Roch Theorem \cite[p. 12]{ACGH}
	\begin{equation*}
		\begin{split}
			\dim\left(\overline{E}\right)&=\deg\left(E\right)-\ell\left(E\right)\\
			&=\left(n+k\right)-\left(k+1\right)=n-1,
		\end{split}
	\end{equation*}
	i.e., $\overline{E}\in\mathbb{G}\left(n-1,g-1\right)$.
	
	Now let's take $D\in C^{(n)}$ such that $D\leq E$ and $\overline{D}=\overline{E}$. Note that $\left(\overline{D}\cdot C\right)=\left(\overline{E}\cdot C\right)\geq E$. Then $\deg\left(\overline{D}\cdot C\right)\geq n+k$, i.e., $D\in R_{n,k}$, so $\overline{E}=\overline{D}\in \mathcal{G}\left(R_{n,k}\right)=:B_{n,k}$. Thus $\beta$ is well defined.
	
	Let's now check that $\beta$ is surjective. Let $\overline{D}\in\mathcal{G}\left(R_{n,k}\right)=:B_{n,k}$. We have $\deg\left(\overline{D}\cdot C\right)\geq n+k$. Then, we can take $E\in C^{(n+k)}$ such that $D\leq E\leq\left(\overline{D}\cdot C\right)$. Note that $\overline{E}=\overline{D}$. By the Geometric Riemann-Roch Theorem \cite[p. 12]{ACGH}	
	\begin{equation*}
		\begin{split}
			\ell\left(E\right)&=\deg\left(E\right)-\dim\left(\overline{E}\right)\\
			&=\left(n+k\right)-\left(n-1\right)=k+1,
		\end{split}
	\end{equation*}
	i.e., $L:=|E|$ is a complete $\mathfrak{g}_{n+k}^k$. Thus $\overline{D}=\overline{E}=\beta\left(E\right)$ and $E\in L\in\mathcal{L}_{n+k}^k\left(C \right)$, so $\beta$ is surjective.
	
	Finally we will prove that the fibers of $\beta$ have finite cardinality. Given $\overline{E}\in B_{n,k}$, we have that $\deg\left(\overline{E}\cdot C\right)\geq n+k$. Let's write $\deg\left(\overline{E}\cdot C\right)=n+l$ for some $l\geq k$. We want to prove that $\beta^{-1}\left(\overline{E}\right)$ is a finite set. Let $F\in\beta^{-1}\left(\overline{E}\right)$, i.e., $\overline{F}=\overline{E}$. Then $F\leq\left(\overline{F}\cdot C\right)=\left(\overline{E}\cdot C\right)$. Therefore the fiber $\beta^{-1}\left(\overline{E}\right)$ has at most $\binom{n+l}{n+k}$ elements.
\end{proof}

The following proposition gives bounds for the dimension of $R_{n,k}$ and of $\overline{B_{n,k}}$.

\begin{prop}\label{lmi4}
	Let $1\leq k\leq n-1$. Assume that $C$ has a complete $\mathfrak{g}_{n+k}^k$. Then
	$$
	\dim\left(C_{n+k}^k\right)=\dim\left(\overline{B_{n,k}}\right)=\dim\left(R_{n,k}\right)\leq n-1
	$$
	Moreover, for each irreducible component $Z$ of $\overline{B_{n,k}}$ we have
	$$
	\max\{k,k+\rho\left(g,k,n+k\right)\}\leq\dim\left(Z\right)\leq n-1.
	$$
\end{prop}

\begin{proof}
	Let $X_1,\dots,X_m$ be the irreducible components of $C_{n+k}^k$. Since $g-\left(n+k\right)+k=g-n\geq2$, by \cite[IV, (1.7) Lemma]{ACGH} we have that no $X_j$ is completely contained in $C_{n +k}^{k+1}$. This implies that $X_j\setminus C_{n+k}^{k+1}\neq\varnothing$ for each $j\in\{1,\dots,m\}$. For each $j\in\{1,\dots,m\}$ we define $Z_j:=X_j\setminus C_{n+k}^{k+1}$. Since $Z_j$ is an open (dense) subset of $X_j$, we have $\dim\left(Z_j\right)=\dim\left(X_j\right)$ for each $j\in\{1,\dots,m \}$. Note that $Z_1,\dots,Z_m$ are the irreducible components of $C_{n+k}^k\setminus C_{n+k}^{k+1}=\bigcup_{L\in\mathcal{L} _{n+k}^k\left(C\right)}L$.
	
	Since $\beta$ is surjective, we have
	$$
	B_{n,k}=\beta\left(Z_1\right)\cup\dots\cup\beta\left(Z_m\right),
	$$
	and the $\beta\left(Z_j\right)$ are irreducible. Taking closure
	$$
	\overline{B_{n,k}}=\overline{\beta\left(Z_1\right)}\cup\dots\cup\overline{\beta\left(Z_m\right)},
	$$
	so that each irreducible component of $\overline{B_{n,k}}$ is one of the $\overline{\beta\left(Z_j\right)}$.
	
	For each $j\in\{1,\dots,m\}$ the restriction $\beta:Z_j\to\overline{\beta\left(Z_j\right)}$ is a dominant morphism of quasi-projective varieties with finite fibers. Then, by the Fiber Dimension Theorem \cite[Theorem 1.25]{SH} for each $j\in\{1,\dots,m\}$ we obtain
	$$
	\dim\left(\overline{\beta\left(Z_j\right)}\right)=\dim\left(Z_j\right).
	$$
	
	Without loss of generality we can assume that $\dim\left(C_{n+k}^k\right)=\dim\left(X_1\right)$. We will prove that $\dim\left(\overline{B_{n,k}}\right)=\dim\left(C_{n+k}^k\right)$. Indeed, let's assume that $\dim\left(\overline{B_{n,k}}\right)=\dim\left(\overline{\beta\left(Z_j\right)}\right)$.
	
	Then
	$$
	\dim\left(\overline{B_{n,k}}\right)=\dim\left(\overline{\beta\left(Z_j\right)}\right)=\dim\left(Z_j\right)=\dim\left(X_j\right)\leq\dim\left(X_1\right)=\dim\left(C_{n+k}^k\right).
	$$
	
	On the other hand,
	$$
	\dim\left(C_{n+k}^k\right)=\dim\left(X_1\right)=\dim\left(Z_1\right)=\dim\left(\overline{\beta\left(Z_1\right)}\right)\leq\dim\left(\overline{B_{n,k}}\right),
	$$
	which proves $\dim\left(\overline{B_{n,k}}\right)=\dim\left(C_{n+k}^k\right)$.
	
	Now we will prove that each irreducible component of $\overline{B_{n,k}}$ has dimension at least $k$. For this it is enough to prove that
	
	$$
	\dim\left(\overline{\beta\left(Z_j\right)}\right)\geq k \ \text{ for all } j\in\{1,\dots,m\}.
	$$
	
	Note that given $L\in\mathcal{L}_{n+k}^k\left(C\right)$, we have that $L\subset Z_j$ for some $j\in\{1,\dots ,m\}$, since $L$ is irreducible. We claim that for every $j\in\{1,\dots,m\}$ there exists $L\in\mathcal{L}_{n+k}^k\left(C\right)$ such that $L\subset Z_j$. By contradiction, assume that $L\not\subset Z_1$ for all $L\in\mathcal{L}_{n+k}^k\left(C\right)$. This implies that for every $L\in\mathcal{L}_{n+k}^k\left(C\right)$ there exists $j\in\{2,\dots,m\}$ such that $L \subset Z_j$. Then
	$$
	Z_1\cup\dots\cup Z_m=\bigcup_{L\in\mathcal{L}_{n+k}^k\left(C\right)}L\subset Z_2\cup\dots\cup Z_m.
	$$
	
	In particular $Z_1\subset Z_2\cup\dots\cup Z_m$, so $Z_1\subset Z_j$ for some $j\in\{2,\dots,m\}$, which is a contradiction. This proves the claim.
	
	Then, given $j\in\{1,\dots,m\}$ there exists $L\in\mathcal{L}_{n+k}^k\left(C\right)$ such that $L\subset Z_j$. This implies that
	$$
	k=\dim\left(L\right)\leq\dim\left(Z_j\right)=\dim\left(\overline{\beta\left(Z_j\right)}\right),
	$$
	so each component of $\overline{B_{n,k}}$ has dimension at least $k$.
	
	Moreover, each component of $C_{n+k}^k$ has dimension greater than or equal to $k+\rho\left(g,k,n+k\right)$ by \cite[p. 159]{ACGH}, so for each $j\in\{1,\dots,m\}$ we have
	$$
	\dim\left(\overline{\beta\left(Z_j\right)}\right)=\dim\left(Z_j\right)=\dim\left(X_j\right)\geq k+\rho\left(g ,k,n+k\right).
	$$
	
	Thus each component $Z$ of $\overline{B_{n,k}}$ satisfies
	$$
	\dim\left(Z\right)\geq\max\{k,k+\rho\left(g,k,n+k\right)\}.
	$$
	
	Now let $Y_1,\dots,Y_r$ be the irreducible components of $R_{n,k}$. Then
	$$
	B_{n,k}:=\mathcal{G}\left(R_{n,k}\right)=\mathcal{G}\left(Y_1\right)\cup\dots\cup\mathcal{G}\left(Y_r\right),
	$$
	where each $\mathcal{G}\left(Y_j\right)$ is irreducible. Taking closure
	$$
	\overline{B_{n,k}}=\overline{\mathcal{G}\left(Y_1\right)}\cup\dots\cup\overline{\mathcal{G}\left(Y_r\right)},
	$$
	so that each irreducible component of $\overline{B_{n,k}}$ is one of the $\overline{\mathcal{G}\left(Y_j\right)}$.

	For each $j\in\{1,\dots,r\}$ the restriction $\mathcal{G}:Y_j\to\overline{\mathcal{G}\left(Y_j\right)}$ is a dominant morphism of quasi-projective varieties with finite fibers. Then, by the Fiber Dimension Theorem \cite[Theorem 1.25]{SH} for each $j\in\{1,\dots,r\}$ we have
	
	$$
	\dim\left(Y_j\right)=\dim\left(\overline{\mathcal{G}\left(Y_j\right)}\right),
	$$
	so $\dim\left(R_{n,k}\right)=\dim\left(\overline{B_{n,k}}\right)=\dim\left(C_{n+k}^k\right)$.
	
	Now, since $R_{n,k}$ is a proper closed subset of $\mathcal{W}_n$, for each $j\in\{1,\dots,r\}$ we have
	$$
	\dim\left(\overline{\mathcal{G}\left(Y_j\right)}\right)=\dim\left(Y_j\right)\leq\dim\left(R_{n,k}\right)\leq\dim\left(\mathcal{W}_n\right)-1=n-1,
	$$
	which also implies that $\dim\left(C_{n+k}^k\right)\leq n-1$. This proves the proposition.
\end{proof}

Here we ask the following questions:

Do we have $\dim\left(Y\right)\geq k$ for each irreducible component $Y$ of $R_{n,k}$? 

Is there an example where an component of $R_{n,k}$ is also an component of $R_{n,k-1}$?

\vspace{0.2cm} 

Now, note that for each $k\in\{1,\dots,n-1\}$
\begin{equation*}
	\begin{split}
		\rho\left(g,k-1,n+k-1\right)-\rho\left(g,k,n+k\right)&=g-k\left(g-n\right)-\left(g-\left(k+1\right)\left(g-n\right)\right)\\
		&=\left(k+1\right)\left(g-n\right)-k\left(g-n\right)\\
		&=g-n\geq2.
	\end{split}
\end{equation*}

Then,
\begin{equation*}
	\begin{split}
		k-1+\rho\left(g,k-1,n+k-1\right)-\left(k+\rho\left(g,k,n+k\right)\right)&=g-n-1\geq1.
	\end{split}
\end{equation*}

Using the previous result we can obtain bounds for the dimension of the multiple locus of $\mathcal{G}$.

\begin{cor}\label{lmi9}
	Let $1\leq k\leq n-1$. Assume that $C$ has a complete $\mathfrak{g}_{n+k}^k$. Then
	$$
	\max\{k,1+\rho\left(g,1,n+1\right)\}\leq\dim\left(R_n\right)\leq n-1.
	$$
\end{cor}

\begin{proof}
	Since $C$ has a complete $\mathfrak{g}_{n+k}^k$, then $R_{n,k}\neq\varnothing$ by Proposition \ref{lmi1}. Then $R_n=R_{n,1}\neq\varnothing$, since $R_{n,k}\subset R_{n,1}$, so
	$$
	\max\{1,1+\rho\left(g,1,n+1\right)\}\leq\dim\left(R_n\right)\leq n-1
	$$
	by Proposition \ref{lmi4}. Moreover,
	$$
	\dim\left(R_n\right)\geq\dim\left(R_{n,k}\right)\geq\max\{k,k+\rho\left(g,k,n+k\right)\}
	$$
	also by Proposition \ref{lmi4}. This proves the corollary.	
\end{proof}

If $C$ has a complete $\mathfrak{g}_{2n-1}^{n-1}$ we can know exactly the dimensions of the $R_{n,k}$ and the $\overline{B_{n,k}}$.

\begin{cor}\label{lmi10}
	If $C$ has a complete $\mathfrak{g}_{2n-1}^{n-1}$, then for all $k\in\{1,\dots,n-1\}$
	$$
	\dim\left(R_{n,k}\right)=\dim\left(\overline{B_{n,k}}\right)=\dim\left(C_{n+k}^k\right)=n-1.
	$$
\end{cor}

\begin{proof}
	For each $k\in\{1,\dots,n-1\}$ we have
	$$
	n-1\leq\max\{n-1,\rho\left(g,n-1,2n-1\right)\}\leq\dim\left(R_{n,n-1}\right)\leq \dim\left(R_{n,k}\right)\leq n-1
	$$
	by Proposition \ref{lmi4}, so that
	$$
	\dim\left(C_{n+k}^k\right)=\dim\left(\overline{B_{n,k}}\right)=\dim\left(R_{n,k}\right)=n-1,
	$$
	also by Proposition \ref{lmi4}. This proves the corollary.
\end{proof}

From Proposition \ref{lmi4} it also follows:

\begin{cor}
	If $C$ has a $\mathfrak{g}_{2n-1}^{n-1}$, then 
	$$
	\dim\left(C_{2n-1}^{n-1}\right)=\dim\left(B_{n,n-1}\right)=\dim\left(R_{n,n-1 }\right)=n-1.
	$$
	Moreover $C_{2n-1}^{n-1}$ and $B_{n,n-1}$ are equidimensional and $C$ has a finite number of $\mathfrak{g}_{2n-1} ^{n-1}$'s.
\end{cor}

\section{The Generalized Torelli Problem}

In this section we will study the Generalized Torelli Problem, i.e., the problem of reconstructing a curve $C\in\mathscr{M}_g$ and families of linear systems on $C$ vía information encoded in the Gauss map on $\mathcal{W}_n$. As mentioned in the introduction, this problem has been studied before for non-hyperelliptic curves defined over $\C$: in \cite{STR1} the case of pencils of linearly equivalent divisors (dimension one), and in \cite{STR2} the case of nets of linearly equivalent divisors (dimension two). There was an attempt to generalize the result of \cite{STR1} by Faucette in his notes \cite{F}, but these contain an error as mentioned in the introduction. In the first subsection of this chapter we are going to generalize the results of Smith and Tapia-Recillas mentioned for the case of linear systems of dimension $k\geq3$ on non-hyperelliptic curves, and without assuming that $W_n$ is smooth. We will also study analogous results for the hyperelliptic case in the second subsection.

\subsection{Non-hyperelliptic case}\label{tg}

We now state the main result of this section, which solves the Generalized Torelli Problem in the non-hyperelliptic case and generalizes the results of \cite{STR1,STR2}.

\begin{thm}\label{tg4}
	Let $1\leq k\leq n-1\leq g-3$, $C\in\mathscr{M}_g$ non-hyperelliptic. Let's assume that $C$ has a $\mathfrak{g}_{n+k}^k$, but not a $\mathfrak{g}_{n+k+1}^{k+1}$. Let $\mathcal{G}:\mathcal{W}_n\to\mathbb{G}\left(n-1,g-1\right)$ be the Gauss map on $\mathcal{W}_n$. Let $L$ be an arbitrary $\mathfrak{g}_{n+k}^k$ on $C$ and consider the morphism $\phi_L:C\to\mathbb{P}^k$ associated to $L$. Then $L$ and the dual hypersurface $\phi_L\left(C\right)^*$ can be reconstructed from the restriction of the Gauss map to the $(n+k)$-intersection locus $R_{n,k}$, i.e., from $\mathcal{G}:R_{n,k}\to B_{n,k}$.
\end{thm}

Here reconstructing means that $B_{n,k}$ contains a copy of each $L=\mathfrak{g}_{n+k}^k$ and a copy of the dual hypersurfaces $\phi_L\left(C\right )^*$, but we can be more specific about how these copies are constructed, as we will see later. In this theorem we don't request that $W_n$ be smooth, as in \cite{STR1} and one case of \cite{STR2}. In fact, we have the following.

\begin{prop}\label{tg10}
	Let $1\leq k\leq n-1\leq g-3$. Suppose that $W_n$ is smooth and that the generic curve $C\in\mathscr{M}_g$ has a $\mathfrak{g}_{n+k}^k$. Then $k\in\{1,2\}$.
\end{prop}

\begin{proof}
	Since $W_n$ is smooth, then $C$ doesn't have a $\mathfrak{g}_n^1$ by Theorem \ref{gwn1}(b). By \cite[V. (1.1)]{ACGH}, this implies that $\rho\left(g,1,n\right)<0$, i.e., $n<\frac{g}{2}+1$. On the other hand, by \cite[V. (1.1)]{ACGH} the generic curve $C\in\mathscr{M}_g$ has a $\mathfrak{g}_{n+k}^k$ if and only if $\rho\left (g,k,n+k\right)\geq0$, which is equivalent to $g\leq\frac{\left(k+1\right)n}{k}$. Using these two inequalities
	$$
	g\leq\frac{\left(k+1\right)n}{k}<\frac{k+1}{k}\left(\frac{g}{2}+1\right)\Rightarrow g\left(k-1\right)<2k+2.
	$$
	
	Since $g\geq k+3$, it follows that
	$$
	\left(k+3\right)\left(k-1\right)<2k+2\Rightarrow k^2<5,
	$$
	which implies $k\in\{1,2\}$.
\end{proof}

The cases $k\in\{1,2\}$, with $W_n$ smooth, were studied in \cite{STR1} and \cite{STR2}. If one wants to study the case $k\geq3$, with $W_n$ smooth, then, by Proposition \ref{tg10}, the generic curve $C\in\mathscr{M}_g$ couldn't have a $\mathfrak {g}_{n+k}^k$ and in fact it isn't clear that there exists a curve $C$ that has a $\mathfrak{g}_{n+k}^k$, but not a $\mathfrak {g}_n^1$, nor a $\mathfrak{g}_{n+k+1}^{k+1}$. 

We ask a question here: Does there exists a $C\in\mathscr{M}_g$ that has a $\mathfrak{g}_{n+k}^k$, but not a $\mathfrak{g}_n^1$ , nor a $\mathfrak{g}_{n+k+1}^{k+1}$, for $3\leq k\leq n-1\leq g-3$?

We will now prove that the hypothesis that the curve $C$ does not have a $\mathfrak{g}_{n+k+1}^{k+1}$ implies that $C$ is non-hyperelliptic. More precisely:

\begin{prop}\label{a1}
	Let $C\in\mathscr{M}_g$ be hyperelliptic and $1\leq n\leq g-1$. Then $C$ has a complete $\mathfrak{g}_{n+k}^k$ for all $k\in\{1,\dots,n\}$.
\end{prop}

\begin{proof}
	If $P_1,\dots,P_{n-k}\in C$ are such that no two of them are conjugated under involution $\iota$, then $k\mathfrak{g}_2^1+P_1+\dots+P_{n-k}$ is a complete $\mathfrak{g}_{n+k}^k$.
\end{proof}

\begin{cor}\label{tg1}
	Let $1\leq k\leq n-1\leq g-3$ and $C\in\mathscr{M}_g$. If $C$ doesn't have a $\mathfrak{g}_{n+k+1}^{k+1}$, then $C$ is non-hyperelliptic.
\end{cor}

\begin{proof}
If $C$ is hyperelliptic, then, by Proposition \ref{a1}, $C$ has a $\mathfrak{g}_{n+k+1}^{k+1}$, since $k+1\leq n$.
\end{proof}

Before proving Theorem \ref{tg4}, we will check that there exist examples of curves $C\in\mathscr{M}_g$ that satisfy the hypotheses of this theorem, i.e., curves $C\in\mathscr{M}_g$ that have a $\mathfrak{g}_{n+k}^k$, but not a $\mathfrak{g}_{n+k+1}^{k+1}$ and therefore non-hyperelliptic by Proposition \ref{tg1}. Note that, removing a point that is not a base point, we have a stratification
$$
\mathscr{M}_g\supset\mathscr{M}_{g,n+1}^1\supset\mathscr{M}_{g,n+2}^2\supset\dots\supset\mathscr{M }_{g,2n-1}^{n-1}\supset\mathscr{M}_{g,2n}^n=\mathscr{H}_g\supset\mathscr{M}_{2n+1} ^{n+1}=\varnothing.
$$
Since $\dim\left(\mathscr{H}_g\right)=2g-1<3g-3=\dim\left(\mathscr{M}_g\right)$, then there must exist $k\in\{1,\dots,n-1\}$ such that $\mathscr{M}_{g,n+k}^k\supsetneqq\mathscr{M}_{g,n+k+1}^{k+ 1}$. To determine these $k$ we will use the Brill-Noether Theory.

\begin{prop}\label{tg15}
	Let $1\leq k\leq n-1\leq g-3$. If $\frac{(k+2)n}{k+1}<g\leq\frac{(k+1)n}{k}$, there exists an open (dense) subset $\mathscr{U}_{n,k}$ of $\mathscr{M}_g$ such that: $C\in\mathscr{M}_g$ is non-hyperelliptic, it has a $\mathfrak{g}_{n+k}^k$, but not a $\mathfrak{g}_{n+k+1}^{k+1}$ if and only if $C\in\mathscr{U}_{n,k}$. 
\end{prop}

\begin{proof}
	The hypothesis $g\leq\frac{(k+1)n}{k}$ is equivalent to $\rho\left(g,k,n+k\right)\geq0$. Then, by \cite[V. (1.1)]{ACGH}, we have $\mathscr{M}_{g,n+k}^k=\mathscr{M}_g$, i.e., each $C\in\mathscr {M}_g$ has a $\mathfrak{g}_{n+k}^k$. On the other hand, the hypothesis $g>\frac{(k+2)n}{k+1}$ is equivalent to $\rho\left(g,k+1,n+k+1\right)<0 $, so, by \cite[V. (1.5)]{ACGH} we have that $\mathscr{M}_{g,n+k+1}^{k+1}$ is a proper closed subset of $\mathscr {M}_g$ possibly empty. Then the open set $\mathscr{U}_{n,k}:=\mathscr{M}_g\setminus\mathscr{M}_{g,n+k+1}^{k+1}$ is non-empty and thus dense. By Proposition \ref{tg1} the open set $\mathscr{U}_{n,k}$ satisfies the required properties.
\end{proof}

From this proposition it follows that for each $g\geq8$, there exist curves $C\in\mathscr{M}_g$ that satisfy the hypotheses of Theorem \ref{tg4} with $k\geq3$.

\begin{cor}
	For all $g\geq8$ there exist $k,n$ with $3\leq k\leq n-1\leq g-3$ such that the variety $\mathscr{M}_{g,n+k}^k\setminus\mathscr{M}_{g,n+k+1}^{k+1}$ is an open (dense) subset of $\mathscr{M}_g$.
\end{cor}

The first of this examples is $k=3$, $n=6$ and $g=8$.

Now we will verify the existence of curves $C\in\mathscr{M}_g$ that satisfy the hypotheses of Theorem \ref{tg4} in the case $g>\tfrac{\left(k+1\right)n}{k}$ with $1\leq k\leq n-1\leq g-3$.

\begin{prop}\label{tg16}
	Let $1\leq k\leq n-1\leq g-3$. Assume that $g>\tfrac{\left(k+1\right)n}{k}$. If $n+k+1\neq g-1$ and $g\leq\tfrac{\left(k+2\right)n-2}{k}$ or if $n+k+1=g-1 $ and $g\leq\tfrac{\left(k+2\right)n-3}{k}$, then there exists a subvariety $\mathscr{V}_{n,k}$ of $\mathscr{M}_g$ of positive codimension and of positive dimension such that: $C\in\mathscr{M}_g$ is non-hyperelliptic, it has a $\mathfrak{g}_{n+k}^k$, but not a $\mathfrak{g} _{n+k+1}^{k+1}$ if and only if $C\in\mathscr{V}_{n,k}$.
\end{prop}

\begin{proof}
The hypothesis $g>\tfrac{\left(k+1\right)n}{k}$ implies that $\rho\left(g,k+1,n+k+1\right)<\rho \left(g,k,n+k\right)<0$. Then, by \cite[V. (1.5)]{ACGH}, $\mathscr{M}_{g,n+k}^k$ and $\mathscr{M}_{g,n+k+1}^{ k+1}$ are proper closed subsets of $\mathscr{M}_g$. Now, the condition $g\leq\tfrac{\left(k+2\right)n-2}{k}$ is equivalent to $-\rho\left(g,k+1,n+k+1\right)\leq g-2$ and the condition $g\leq\tfrac{\left(k+2\right)n-3}{k}$ is equivalent to $-\rho\left(g,k+1 ,n+k+1\right)\leq g-3$. Therefore, if $n+k+1\neq g-1$ and $g\leq\tfrac{\left(k+2\right)n-2}{k}$ or if $n+k+1= g-1$ and $g\leq\tfrac{\left(k+2\right)n-3}{k}$, we can apply \cite[Theorem 2.1]{MT} to obtain that $\mathscr{M }_{g,n+k+1}^{k+1}$ has at least one irreducible component of dimension $3g-3+\rho\left(g,k+1,n+k+1\right)$. On the other hand, by \cite[Theorem 0.1]{S} each irreducible component of $\mathscr{M}_{g,n+k}^k$ is of dimension $\geq3g-3+\rho\left(g ,k,n+k\right)>3g-3+\rho\left(g,k+1,n+k+1\right)$. Thus $\mathscr{M}_{g,n+k+1}^{k+1}\subsetneqq\mathscr{M}_{g,n+k}^k$. Therefore the variety $\mathscr{V}_{n,k}:=\mathscr{M}_{g,n+k}^k\setminus\mathscr{M}_{g,n+k+1}^{k+1} $ satisfies the requirements of the proposition.
\end{proof}

From this proposition we deduce the following result.

\begin{cor}
	For all $g\geq7$ there exist $k,n$ with $3\leq k\leq n-1\leq g-3$ such that the variety $\mathscr{M}_{g,n+k}^k\setminus\mathscr{M}_{g,n+k+1}^{k+1}$ is of positive codimension in $\mathscr{M}_g$ and of positive dimension.
\end{cor}

The first of this examples is $k=3$, $n=5$ and $g=7$. Now we will prove that the linear systems we will work with will always be complete.

\begin{prop}
	Let $1\leq k\leq n-1\leq g-3$. Suppose $C$ has a $\mathfrak{g}_{n+k}^k$, but not a $\mathfrak{g}_{n+k+1}^{k+1}$. Then every $\mathfrak{g}_{n+k}^k$ in $C$ is complete.
\end{prop}

\begin{proof}
	By contradiction, assume that $C$ has a non complete $\mathfrak{g}_{n+k}^k$, say $L$. Then, given $D\in L$, we have that $|E|$ is a $\mathfrak{g}_{n+k}^l$ for some $l\geq k+1$. We can take $P_1,\dots,P_{l-k}\in C$ general points such that $|D+P_1+\dots+P_{l-k}|$ is a complete $\mathfrak{g}_{n+l}^l$. Then $R_{n,k+1}\supset R_{n,l}\neq\varnothing$ by Proposition \ref{lmi1}, so by this same proposition $C$ must have a $\mathfrak{g}_{n+k+1} ^{k+1}$, which contradicts the initial hypothesis. 
\end{proof}

Under the hypothesis of this proposition, the $R_{n,k}$ take a more particular form.

\begin{prop}\label{tg2}
	Let $1\leq k\leq n-1\leq g-3$ and $C\in\mathscr{M}_g$. Assume that $C$ has a $\mathfrak{g}_{n+k}^k$, but not a $\mathfrak{g}_{n+k+1}^{k+1}$. Then $R_{n,k+1}=\varnothing$ and
	$$
	R_{n,k}=\{D\in\mathcal{W}_n:\deg\left(\overline{D}\cdot C\right)=n+k\}.
	$$
\end{prop}

\begin{proof}
	The fact that $R_{n,k+1}=\varnothing$ is obtained from Proposition \ref{lmi1}, since $C$ does not have a $\mathfrak{g}_{n+k+1}^{k +1}$. Now, let $D\in R_{n,k}$, so that $\deg\left(\overline{D}\cdot C\right)\geq n+k$. We will prove that $\deg\left(\overline{D}\cdot C\right)=n+k$. By contradiction let's assume that $\deg\left(\overline{D}\cdot C\right)\geq n+k+1$. Then we can take $E\in C^{(n+k+1)}$ such that $D\leq E\leq\left(\overline{D}\cdot C\right)$. Note that $\overline{E}=\overline{D}$. Then, by the Geometric Riemann-Roch Theorem 
	\begin{equation*}
		\begin{split}
			\ell\left(E\right)&=\deg\left(E\right)-\dim\left(\overline{E}\right)=\left(n+k\right)-\left(n-1\right)=k+1,
		\end{split}
	\end{equation*}
	so $|E|$ is a complete $\mathfrak{g}_{n+k+1}^{k+1}$, but this contradicts the hypothesis of the proposition. Therefore $\deg\left(\overline{D}\cdot C\right)=n+k$ must be satisfied. This proves the proposition.	
\end{proof}

Let $1\leq k\leq n-1\leq g-3$ and $C\in\mathscr{M}_g$ non-hyperelliptic. If $C$ has a $\mathfrak{g}_{n+k}^k$, but not a $\mathfrak{g}_{n+k+1}^{k+1}$, then we have a stratification
$$
\mathcal{W}_n=R_{n,0}\supset R_n=R_{n,1}\supset R_{n,2}\supset\dots\supset R_{n,k}.
$$

Remember that for $0\leq k\leq n-1$ we defined $B_{n,k}:=\mathcal{G}\left(R_{n,k}\right)\subset\mathbb{G}\left(n-1,g-1\right) $. Note that
$$
\mathcal{G}\left(\mathcal{W}_n\right)=B_{n,0}\supset \mathcal{G}\left(R_n\right)=B_{n,1}\supset B_{n,2}\supset\dots\supset B_{n,k}.
$$

Under the hypothesis of the previous proposition we can obtain an intrinsic description of $B_{n,k}$ that not depends on the Gauss map.

\begin{prop}\label{tg13}
	Let $1\leq k\leq n-1\leq g-3$ and $C\in\mathscr{M}_g$. Assume that $C$ has a $\mathfrak{g}_{n+k}^k$, but not a $\mathfrak{g}_{n+k+1}^{k+1}$. Then
	$$
	B_{n,k}=\{W\in\mathbb{G}\left(n-1,g-1\right):\deg\left(W\cdot C\right)=n+k\}.
	$$
\end{prop}

\begin{proof}
	Define $L_{n,k}:=\{W\in\mathbb{G}\left(n-1,g-1\right):\deg\left(W\cdot C\right)=n+k\}$. If $D\in R_{n,k}$, then $\deg\left(\overline{D}\cdot C\right)=n+k$ by Proposition \ref{tg2}, so $\overline{D}\in L_{n,k}$, which shows $B_{n,k}\subset L_{n,k}$.
	
	Conversely, let $W\in L_{n,k}$, that is, $W\in\mathbb{G}\left(n-1,g-1\right)$ satisfies $\deg\left(W\cdot C\right) =n+k$. Let $E:=\left(W\cdot C\right)$. Since $\overline{E}\subset W$, we have $\dim\left(\overline{E}\right)\leq\dim\left(W\right)=n-1$. We claim that $\dim\left(\overline{E}\right)=n-1$. We will prove it by contradiction. Let's assume that $\dim\left(\overline{E}\right)<n-1$. Then, by the Geometric Riemann-Roch Theorem
	\begin{equation*}
		\begin{split}
			\ell\left(E\right)&=\deg\left(E\right)-\dim\left(\overline{E}\right)>\left(n+k\right)-\left(n-1\right)=k+1,
		\end{split}
	\end{equation*}
	but this implies that $C$ has a $\mathfrak{g}_{n+k+1}^{k+1}$, which contradicts the hypothesis of the proposition. Therefore we must have that $\dim\left(\overline{E}\right)=n-1$.
	
	Now, let's take $D\in C^{(n)}$ such that $D\leq E$ and $\overline{D}=\overline{E}=W$. Then $D\in R_{n,k}$, because $\deg\left(\overline{D}\cdot C\right)=\deg\left(W\cdot C\right)=n+k$, so $W=\overline{D}=\mathcal{G}\left(D\right)\in\mathcal{G}\left(R_{n,k}\right)=:B_{n,k}$.
\end{proof}

Through the following proposition we will begin to relate the set $B_{n,k}$ (then we will see that is a variety) with all the linear systems $\mathfrak{g}_{n+k}^k$ on $C$ assuming that $C$ doesn't have a $\mathfrak{g}_{n+k+1}^{k+1}$.

\begin{prop}\label{tg12}
	Let $1\leq k\leq n-1\leq g-3$ and $C\in\mathscr{M}_g$. Assume that $C$ has a $\mathfrak{g}_{n+k}^k$, but not a $\mathfrak{g}_{n+k+1}^{k+1}$. Then the map
	$$
	\beta:C_{n+k}^k\to B_{n,k}, \ E\mapsto\overline{E}
	$$
	is a regular bijection.
\end{prop}

\begin{proof}
	By Proposition \ref{lmi8} we have that $\beta$ is well defined, surjective and regular. So we just need to prove that $\beta$ is injective. Let's assume that $E,F\in C_{n+k}^k$ satisfy that $\beta\left(E\right)=\beta\left(F\right)$, i.e., $\overline{E}= \overline{F}$. Then, since $\deg\left(\overline{E}\cdot C\right)=\deg\left(\overline{F}\cdot C\right)=n+k$, we have
	$$
	E=\left(\overline{E}\cdot C\right)=\left(\overline{F}\cdot C\right)=F,
	$$
	so $\beta$ is injective.
\end{proof}

Note that under the hypothesis of Proposition \ref{tg12} we have that $B_{n,k}$ is a projective variety. The next proposition shows that $B_{n,k}$ is a disjoint union of copies of all the $\mathfrak{g}_{n+k}^k$ on $C$ via $\beta$, assuming that $C$ doesn's have a $\mathfrak{g}_{n+k+1}^{k+1}$.

\begin{prop}\label{tg3}
	Let $1\leq k\leq n-1\leq g-3$ and $C\in\mathscr{M}_g$. Assume that $C$ has a $\mathfrak{g}_{n+k}^k$, but not a $\mathfrak{g}_{n+k+1}^{k+1}$. Then $B_{n,k}$ is in an unique way a disjoint union of a family of closed subvarieties of $B_{n,k}$ birational to $\mathbb{P}^k$; more precisely any such subvariety of $B_{n,k}$ is of the form $\beta\left(L\right)$ with $L\in\mathcal{L}_{n+k}^k\left(C\right)$.
\end{prop}

\begin{proof}
	Let $X$ be a closed subvariety of $B_{n,k}$ birational to $\mathbb{P}^k$. Then $Z:=\beta^{-1}\left(X\right)\subset C_{n+k}^k$ is closed and birational to $\mathbb{P}^k$. Let's consider the map
	$$
	\rho:C^{(n+k)}\to W_{n+k}\subset\Pic^{n+k}\left(C\right), \ D\mapsto\mathscr{O}_C\left (D\right).
	$$
	By \cite[V, Corollary 3.9]{CS} its restriction to $Z$ must be constant, i.e., $\rho\left(Z\right)=\{\rho\left(E\right)\}$ for some divisor $E\in Z\subset C_{n+k}^k$. We claim that $Z=|E|$.
	
	Given $F\in Z$, we have $\rho\left(F\right)=\rho\left(E\right)$, i.e., $F\sim E$, so $F\in|E| $. This shows that $Z\subset|E|$. Now since $Z$ is closed, $|E|$ is irreducible, $Z\subset|E|$ and both are of the same dimension, then we must have that $Z=|E|$. Thus $X=\beta\left(|E|\right)$, with $|E|\in\mathcal{L}_{n+k}^k\left(C\right)$. This proves the proposition.
\end{proof}

From this proposition we deduce the following.

\begin{cor}\label{tg5}
	Let $1\leq k\leq n-1\leq g-3$ and $C\in\mathscr{M}_g$ non-hyperelliptic. Let's assume that $C$ has a $\mathfrak{g}_{n+k}^k$, but not a $\mathfrak{g}_{n+k+1}^{k+1}$. Then we have a bijection
	\begin{equation*}
		\begin{split}
			\mathcal{L}_{n+k}^k\left(C\right)&\to\{X\subset B_{n,k}:X\text{ is a closed subvariety of }B_{n,k}\text{ birational to }\mathbb{P }^k\}\\
			L&\mapsto\beta\left(L\right).
		\end{split}
	\end{equation*}
	Moreover, $L\neq M$ if and only if $\beta\left(L\right)\cap\beta\left(M\right)=\varnothing$.
\end{cor}

\begin{proof}
	This is a direct consequence of Proposition \ref{tg3}.
\end{proof}

Now we will prove Theorem \ref{tg4}.

\begin{proof}[Proof of Theorem \ref{tg4}]
	By Corollary \ref{tg5}, every linear system $L\in\mathcal{L}_{n+k}^k\left(C\right)$ can be reconstructed from a unique closed subvariety of $B_{n,k}$ birational to $\mathbb{P}^k$, namely $\beta\left(L\right)$. This proves the first part of the theorem.
	
	Now, let $L\in\mathcal{L}_{n+k}^k\left(C\right)$ and consider the morphism $\phi_L:C\to\mathbb{P}^k$ associated to $ L$. We will prove that the dual hypersurface $\phi_L\left(C\right)^*$ can be reconstructed from the restriction of the Gauss map to the $(n+k)$-intersection locus $R_{n,k}$, i.e., from $\mathcal{G}:R_{n,k}\to B_{n,k}$.
	
	Let's first assume that $L$ is base point free. In this case each $E\in L$ can be written as $E=\phi_L^*H$ for a unique hyperplane $H$ in $\mathbb{P}^k$, i.e.,
	$$
	L=\{\phi_L^*H:H\text{ hyperplane in }\mathbb{P}^k\}.
	$$
	
	In this way we obtain
	\begin{equation*}
		\begin{split}
			\beta\left(L\right)&=\{\overline{E}\in\mathbb{G}\left(n-1,g-1\right):E\in L\}\\
			&=\{\overline{\phi_L^*H}\in\mathbb{G}\left(n-1,g-1\right):H\text{ hyperplane in }\mathbb{P}^k\}.
		\end{split}
	\end{equation*}
	
	Let's now consider
	\begin{equation*}
		\begin{split}
			\beta\left(L\setminus L_\red\right)&=\{\overline{\phi_L^*H}\in\beta\left(L\right):\phi_L^*H\text{ is non-reduced }\}\\
			&=\{\overline{\phi_L^*H}\in\beta\left(L\right):\overline{\phi_L^*H}\text{ is tangent to }C\}\\
			&=\{\overline{\phi_L^*H}\in\beta\left(L\right):H\text{ is a hyperplane tangent to }\phi_L\left(C\right)\}.
		\end{split}
	\end{equation*}
	
	Using Bertini's Theorem and the fact that $\beta$ is a projective morphism, we obtain that $\beta\left(L\setminus L_\red\right)$ is a proper closed subset of $\beta \left(L\right)$. The map
	$$
	\phi_L\left(C\right)^*\to\beta\left(L\setminus L_\red\right), \ H\mapsto\overline{\phi_L^*H}
	$$
	is a regular bijection, so $\beta\left(L\right)\subset B_{n,k}$ contains a copy of the dual hypersurface $\phi_L\left(C\right)^*$, namely $\beta\left(L\setminus L_\red\right)$.
	
	Now suppose that $L$ has base points and let $B$ be its base locus. We define $\phi_L:=\phi_{L-B}$. Note that each $E\in L$ can be written as $E=\phi_L^*H+B$ for a unique hyperplane $H$ in $\mathbb{P}^k$, i.e.,
	$$
	L=\{\phi_L^*H+B:H\text{ hyperplane in }\mathbb{P}^k\}.
	$$
	
	As in the previous case, we have
	\begin{equation*}
		\begin{split}
			\beta\left(L\right)&=\{\overline{E}\in\mathbb{G}\left(n-1,g-1\right):E\in L\}\\
			&=\{\overline{\phi_L^*H+B}\in\mathbb{G}\left(n-1,g-1\right):H\text{ hyperplane in }\mathbb{P}^k\}.
		\end{split}
	\end{equation*}
	
	Let's now consider
	\begin{equation*}
		\begin{split}
			\beta\left(L\setminus L_\red\right)&=\{\overline{\phi_L^*H+B}\in\beta\left(L\right):\phi_L^*H\text{ is non-reduced}\}\\
			&=\{\overline{\phi_L^*H+B}\in\beta\left(L\right):H\text{ is a hyperplane tangent to }\phi_L\left(C\right)\}.
		\end{split}
	\end{equation*}
	
	As in the previous case, using Bertini's Theorem and the fact that $\beta$ is a projective morphism, we obtain that $\beta\left(L\setminus L_\red\right)$ is a proper closed subset of $\beta\left(L\right)$. Similarly to the previous case, the map
	$$
	\phi_L\left(C\right)^*\to\beta\left(L\setminus L_\red\right), \ H\mapsto\overline{\phi_L^*H+B}
	$$
	is a regular bijection, so $\beta\left(L\right)\subset B_{n,k}$ contains a copy of the dual hypersurface $\phi_L\left(C\right)^*$, namely $\beta\left(L\setminus L_\red\right)$. This completes the proof of Theorem \ref{tg4}.
\end{proof}

Summary of the problem studied: We take $1\leq k\leq n-1\leq g-3$ and a curve $C\in\mathscr{M}_g$ such that it has a $\mathfrak{g}_{n+k}^k$, but not a $\mathfrak{g}_{n+k+1}^{k+1}$. This implies that $C$ is non-hyperelliptic. We identify $C$ with its image via the canonical morphism. The set $\{W\in\mathbb{G}\left(n-1,g-1\right):\left(W\cdot C\right)=n+k\}$ turns out to be $B_{n ,k}$. If $X$ is a rational closed subvariety of $B_{n,k}$ of dimension $k$ the set $\{\left(W\cdot C\right):W\in X\}$ is a $\mathfrak{g}_{n+k}^k$, say $L$. Every $\mathfrak{g}_{n+k}^k$ in $C$ is obtained this way. Moreover $X$ contains a copy of the dual hypersurface $\phi_L\left(C\right)^*$ namely $\{W\in X:\left(W\cdot C\right)\in L\setminus L_ \red\}=\beta\left(L\setminus L_\red\right)$, where $\phi_L:C\to\mathbb{P}^k$ is the morphism associated to $L$. In this way, each $L=\mathfrak{g}_{n+k}^k$ and the dual hypersurface $\phi_L\left(C\right)^*$ can be reconstructed through information encoded in the rational and closed subvarieties of dimension $k$ of $B_{n,k}$. This remains to be done for concrete examples.

\subsection{Hyperelliptic case}\label{tg0}

In this subsection we will study some results analogous to those of the previous subsection, that relate linear systems on a hyperelliptic curve $C\in\mathscr{M}_g$ to the image of the Gauss map. In this case $C$ has a complete $\mathfrak{g}_{n+k}^k$ for every $1\leq k\leq n\leq g-1$ by Proposition \ref{a1}. Now we characterize all the complete $\mathfrak{g}_{n+k}^k$'s on a hyperelliptic curve $C\in\mathscr{M}_g$.

\begin{prop}\label{tg6}
	Let $C\in\mathscr{M}_g$ be hyperelliptic and $1\leq k\leq n\leq g-1$. Then each complete $\mathfrak{g}_{n+k}^k$ is of the form
	$$
	k\mathfrak{g}_2^1+P_1+\dots+P_{n-k},
	$$
	where no two of the $P_1,\dots,P_{n-k}$ are conjugate under the hyperelliptic involution $\iota$.
\end{prop}

\begin{proof}
	Let $|E|$ be a complete $\mathfrak{g}_{n+k}^k$, with $E$ effective. By \cite[I. D-9]{ACGH} it is enough to prove that $E$ is special. This is obtained by the numerical conditions and the Riemann-Roch Theorem.
\end{proof}

Now we want to relate the image of the Gauss map, i.e., $\mathcal{G}\left(\mathcal{W}_n\right)$, with the complete linear systems $L=\mathfrak{g}_{n+k }^k$ on $C$. We begin this study with the following proposition.

\begin{prop}\label{tg11}
	Let $C\in\mathscr{M}_g$ be hyperelliptic and $1\leq k\leq n\leq g-1$. Then the map
	$$
	\beta:\bigcup_{L\in\mathcal{L}_{n+k}^k\left(C\right)}L\to\mathbb{G}\left(n-1,g-1\right ), \ F\mapsto\overline{\phi\left(F\right)}
	$$
	is regular. Moreover, for each $L\in\mathcal{L}_{n+k}^k\left(C\right)$, the restriction $\beta|_L:L\to\beta\left(L\right )$ is bijective.
\end{prop}

\begin{proof}
	First we will check that $\beta$ is well defined. Given $F\in L\in\mathcal{L}_{n+k}^k\left(C\right)$, we have $\ell\left(F\right)=k+1$. Then, by the Geometric Riemann-Roch Theorem
	\begin{equation*}
		\begin{split}
			\dim\left(\overline{\phi\left(F\right)}\right)&=\deg\left(F\right)-\ell\left(F\right)\\
			&=\left(n+k\right)-\left(k+1\right)=n-1,
		\end{split}
	\end{equation*}
	so $\beta\left(F\right)\in\mathbb{G}\left(n-1,g-1\right)$. This proves that $\beta$ is well defined.
	
	Now let $L\in\mathcal{L}_{n+k}^k\left(C\right)$. We will prove that $\beta|_L:L\to\beta\left(L\right)$ is bijective, for which it is enough to prove that it is injective. By Proposition \ref{tg6} the linear system $L$ is of the form
	$$
	L=k\mathfrak{g}_2^1+P_1+\dots+P_{n-k},
	$$
	where no two of the $P_1,\dots,P_{n-k}$ are conjugate under the involution $\iota$. Now assume that $F_1,F_2\in L$ satisfy $\beta\left(F_1\right)=\beta\left(F_2\right)$, i.e., $\overline{\phi\left(F_1\right)} =\overline{\phi\left(F_2\right)}$. Let's write 
	$$
	F_1=p_1+\iota\left(p_1\right)+\dots+p_k+\iota\left(p_k\right)+P_1+\dots+P_{n-k},
	$$
	$$
	F_2=q_1+\iota\left(q_1\right)+\dots+q_k+\iota\left(q_k\right)+P_1+\dots+P_{n-k}.
	$$
	
	From the equality $\overline{\phi\left(F_1\right)}=\overline{\phi\left(F_2\right)}$, we deduce
	$$
	\left(\phi\left(p_1\right),\dots,\phi\left(p_k\right)\right)=\left(\phi\left(q_{j_1}\right),\dots,\phi \left(q_{j_k}\right)\right)
	$$
	for some $j_i$ such that $\{j_1,\dots,j_k\}=\{1,\dots,k\}$. Relabeling if necessary, without loss of generality we can assume that
	$$
	\left(\phi\left(p_1\right),\dots,\phi\left(p_k\right)\right)=\left(\phi\left(q_1\right),\dots,\phi\left( q_k\right)\right).
	$$
	
	From this it follows that $\{p_j,\iota\left(p_j\right)\}=\{q_j,\iota\left(q_j\right)\}$ for each $j\in\{1,\dots ,k\}$, which implies
	$$
	p_j+\iota\left(p_j\right)=q_j+\iota\left(q_j\right),
	$$
	for all $j\in\{1,\dots,k\}$. Thus $F_1=F_2$, which proves that $\beta|_L:L\to\beta\left(L\right)$ is injective.
\end{proof}

In general the image of $\beta$ can be larger than the image of the Gauss map $\mathcal{G}\left(\mathcal{W}_n\right)$. To characterize $\mathcal{G}\left(\mathcal{W}_n\right)$ in terms of the $L=\mathfrak{g}_{n+k}^k$ and $\beta$, we need to restrict a little these $L$. We will do this through the following definition. 

Let $C\in\mathscr{M}_g$ be hyperelliptic and $r,d$ be integers, with $r\geq0$, $d\geq1$. Given a special linear system $L\in\mathcal{L}_d^r\left(C\right)$, by \cite[I. D-9]{ACGH}, it must be of the form
$$
L=r\mathfrak{g}_2^1+P_1+\dots+P_{n-2r},
$$
where no two of the $P_1,\dots,P_{n-2r}$ are conjugate under the involution $\iota$. Let $B:=P_1+\dots+P_{n-2r}$.

We define $L_\nc$ as the set of divisors $Q_1+\iota\left(Q_1\right)+\dots+Q_r+\iota\left(Q_r\right)+B\in L$ such that for each $i\in\{1,\dots,r\}$ exists $q_i\in\{Q_i,\iota\left(Q_i\right)\}$ with the property that no two of the $q_1 ,\dots,q_r,P_1,\dots,P_{n-2r}$ are conjugate under the hyperelliptic involution $\iota$.

Here we ask a question: is $L_\nc$ open in $L$? 

The following result proves that $L_\nc$ is dense in $L$.

\begin{lema}
	Let $C\in\mathscr{M}_g$ be hyperelliptic and $r,d$ be integers with $r\geq0,d\geq1$. Given a special linear system $L\in\mathcal{L}_d^r\left(C\right)$, we have that $L_\red\subset L_\nc$.
\end{lema}

\begin{proof}
	By \cite[I. D-9]{ACGH}, the linear system $L$ must be of the form
	$$
	L=r\mathfrak{g}_2^1+P_1+\dots+P_{n-2r},
	$$
	where no two of the $P_1,\dots,P_{n-2r}$ are conjugate under the involution $\iota$.
	
	Let $F=Q_1+\iota\left(Q_1\right)+\dots+Q_r+\iota\left(Q_r\right)+P_1+\dots+P_{n-2r}\in L_\red$. We will prove that $F\in L_\nc$. Since $F\in L_\red$, we have that $Q_1+\iota\left(Q_1\right)+\dots+Q_r+\iota\left(Q_r\right)$ is reduced. Let $q_i\in\{Q_i,\iota\left(Q_i\right)\}$ for each $i\in\{1,\dots,r\}$. It is enought to prove that no two of the $q_1,\dots,q_r,P_1,\dots,P_{n-2r}$ are conjugate under $\iota$. Since $F$ is reduced, then no two of the $q_1,\dots,q_r$ are conjugate under $\iota$. Then it is enough to prove that $q_i\neq\iota\left(P_j\right)$ for all $i\in\{1,\dots,r\}$ and all $j\in\{1,\dots,n- 2r\}$. But if $q_i=\iota\left(P_j\right)$ for some $i,j$, then $\iota\left(q_i\right)=P_j$, so $F$ is not reduced, which is a contradiction.
\end{proof}

Note that $L_\red\subset L_\nc$ implies that $L=\overline{L_\red}\subset\overline{L_\nc}\subset\overline{L}=L$, i.e., $L_\nc $ is dense in $L$. Now we characterize the image of the Gauss map, i.e., $\mathcal{G}\left(\mathcal{W}_n\right)$, in terms of $\beta$ and the $L_\nc$.

\begin{prop}\label{tg7}
	Let $C\in\mathscr{M}_g$ be hyperelliptic and $1\leq k\leq n\leq g-1$. Consider the morphism $\beta$ of Proposition \ref{tg11}. Then we have that
	$$
	\bigcup_{L\in\mathcal{L}_{n+k}^k\left(C\right)}\beta\left(L_\nc\right)=\mathcal{G}\left(\mathcal{W}_n\right).
	$$
\end{prop}

\begin{proof}
	Let $F\in L_\nc$, for a $L\in\mathcal{L}_{n+k}^k\left(C\right)$. We will prove that $\beta\left(F\right)\in\mathcal{G}\left(\mathcal{W}_n\right)$. By Proposition \ref{tg6} the linear system $L$ is of the form
	$$
	L=k\mathfrak{g}_2^1+P_1+\dots+P_{n-k},
	$$
	where no two of the $P_1,\dots,P_{n-k}$ are conjugate under the involution $\iota$. Let's write
	$$
	F=Q_1+\iota\left(Q_1\right)+\dots+Q_k+\iota\left(Q_k\right)+P_1+\dots+P_{n-k}.
	$$
	Since $F\in L_\nc$, we can assume that no two of the $Q_1,\dots,Q_k,P_1,\dots,P_{n-k}$ are conjugate under $\iota$. This implies that $Q_1+\dots+Q_k+P_1+\dots+P_{n-k}\in\mathcal{G}\left(\mathcal{W}_n\right)$. Then
	
	$$
	n-1=\dim\left(\overline{\phi\left(Q_1+\dots+Q_k+P_1+\dots+P_{n-k}\right)}\right)\leq\dim\left(\overline{\phi\left(F\right)}\right)=n-1,
	$$
	which implies
	\begin{equation*}
		\begin{split}
			\beta\left(F\right)=\overline{\phi\left(F\right)}&=\overline{\phi\left(Q_1+\dots+Q_k+P_1+\dots+P_{n-k}\right)}\\
			&=\mathcal{G}\left(Q_1+\dots+Q_k+P_1+\dots+P_{n-k}\right)\in\mathcal{G}\left(\mathcal{W}_n\right).
		\end{split}
	\end{equation*}
	
	This proves that $\beta\left(L_\nc\right)\subset\mathcal{G}\left(\mathcal{W}_n\right)$.
	
	To prove the other inclusion, let $W\in\mathcal{G}\left(\mathcal{W}_n\right)$, i.e., $W=\overline{\phi\left(D\right)}$ for some $D\in\mathcal{W}_n$. Let's write $D=p_1+\dots+p_n$. Since $D\in\mathcal{W}_n$, then no two of the $p_1,\dots,p_n$ are conjugate under $\iota$. Let $F:=p_1+\iota\left(p_1\right)+\dots+p_k+\iota\left(p_k\right)+p_{k+1}+\dots+p_n$, and let
	$$
	L:=|F|=k\mathfrak{g}_2^1+p_{k+1}+\dots+p_n\in\mathcal{L}_{n+k}^k\left(C\right).
	$$
	
	Note that $F\in L_\nc$. Moreover $W=\overline{\phi\left(D\right)}=\overline{\phi\left(F\right)}=\beta\left(F\right)\in\beta\left(L_\nc\right)$.
\end{proof}

The following example shows that the image of $\beta$ (without restricting the $L$ to $L_\nc$) can be larger than the image of the Gauss map.

\begin{ej}
	Let $k=2$, $n=4$, $g=6$ and $C\in\mathscr{M}_g$ hyperelliptic. We will show an example of a $W\in\beta\left(L\right)$ such that $W\notin\mathcal{G}\left(\mathcal{W}_n\right)$. Let $p\in C$ be a fixed point of $\iota$. Let $P_1,P_2\in C$ be such that they aren't conjugate under $\iota$. Then $L:=|4p|+P_1+P_2$ is a complete $\mathfrak{g}_6^2$. Let $F:=4p+P_1+P_2\in L$. Then $W:=\beta\left(F\right)=\overline{\phi\left(F\right)}\in\mathbb{G}\left(3.5\right)$, but $W\notin\mathcal{G}\left(\mathcal{W}_4\right)$, since $|2p|+P_1+P_2$ is a $\mathfrak{g}_4^1$.
\end{ej}

The following corollary characterizes the closure of the image of the Gauss map in terms of the image of $\beta$ (without restricting the $L$ to $L_\nc$).

\begin{cor}
	Let $C\in\mathscr{M}_g$ be hyperelliptic and $1\leq k\leq n\leq g-1$. We have that
	$$
	\overline{\mathcal{G}\left(\mathcal{W}_n\right)}=\overline{\bigcup_{L\in\mathcal{L}_{n+k}^k\left(C\right)}\beta\left(L\right)}.
	$$
\end{cor}

\begin{proof}
	By Proposition \ref{tg7}
	$$
	\mathcal{G}\left(\mathcal{W}_n\right)=\bigcup_{L\in\mathcal{L}_{n+k}^k\left(C\right)}\beta\left(L_\nc\right)\subset\bigcup_{L\in\mathcal{L}_{n+k}^k\left(C\right)}\beta\left(L\right).
	$$
	
	Then, taking closure
	$$
	\overline{\mathcal{G}\left(\mathcal{W}_n\right)}\subset\overline{\bigcup_{L\in\mathcal{L}_{n+k}^k\left(C\right)}\beta\left(L\right)}
	$$
	
	On the other hand, for each $L\in\mathcal{L}_{n+k}^k\left(C\right)$ we have $\overline{\beta\left(L_\nc\right)}=\overline{\beta\left(\overline{L_\nc}\right)}=\overline{\beta\left(L\right)}=\beta\left(L\right)$. Then
	\begin{equation*}
		\begin{split}
			\overline{\mathcal{G}\left(\mathcal{W}_n\right)}=\overline{\bigcup_{L\in\mathcal{L}_{n+k}^k\left(C\right)}\beta\left(L_\nc\right)}\supset\bigcup_{L\in\mathcal{L}_{n+k}^k\left(C\right)}\overline{\beta\left(L_\nc\right)}=\bigcup_{L\in\mathcal{L}_{n+k}^k\left(C\right)}\beta\left(L\right),
		\end{split}
	\end{equation*}
	so, taking closure
	$$
	\overline{\mathcal{G}\left(\mathcal{W}_n\right)}\supset\overline{\bigcup_{L\in\mathcal{L}_{n+k}^k\left(C\right)}\beta\left(L\right)},
	$$
	which proves the other inclusion.
\end{proof}

The following result is an analogue of Theorem \ref{tg4}, but for the hyperelliptic case.

\begin{thm}\label{tg8}
	Let $C\in\mathscr{M}_g$ be hyperelliptic, $1\leq k\leq n\leq g-1$. Let $\mathcal{G}:\mathcal{W}_n\to\mathbb{G}\left(n-1,g-1\right)$ be the Gauss map on $\mathcal{W}_n$. Then the image $\mathcal{G}\left(\mathcal{W}_n\right)$ is a union of sets whose closures are projective rational varieties of dimension $k$. Moreover, given $L\in\mathcal{L}_{n+k}^k\left(C\right)$, consider the map $\phi_L:C\to\mathbb{P}^k$ associated to $L$. Then $\beta\left(L\right)\subset\overline{\mathcal{G}\left(\mathcal{W}_n\right)}$ contains a copy of the dual hypersurface $\phi_L\left(C \right)^*$.
\end{thm}

\begin{proof}
	The first part follows from Proposition \ref{tg7} and from the fact that
	$$
	\overline{\beta\left(L_\nc\right)}=\overline{\beta\left(\overline{L_\nc}\right)}=\overline{\beta\left(L\right)}= \beta\left(L\right)
	$$
	for each $L\in\mathcal{L}_{n+k}^k\left(C\right)$. This proves the first part of the theorem.
	
	Now, let $L\in\mathcal{L}_{n+k}^k\left(C\right)$. We will prove that $\beta\left(L\right)$ contains a copy of the dual hypersurface $\phi_L\left(C\right)^*$. By Proposition \ref{tg6} this $L$ is of the form
	$$
	L=k\mathfrak{g}_2^1+P_1+\dots+P_{n-k},
	$$
	where no two of the $P_1,\dots,P_{n-k}\in C$ are conjugate under the involution $\iota$. Let's write down $B:=P_1+\dots+P_{n-k}$.
	
	Each $F\in k\mathfrak{g}_2^1$ can be expressed as $F=\phi_L^*H$ for a unique hyperplane $H$ in $\mathbb{P}^k$. Then
	$$
	k\mathfrak{g}_2^1=\{\phi_L^*H:H\text{ hyperplane in }\mathbb{P}^k\}.
	$$
	Thus
	\begin{equation*}
		\begin{split}
			\beta\left(L\right)&=\{\overline{\phi\left(F+B\right)}:F\in k\mathfrak{g}_2^1\}=\{\overline{\phi\left(\phi_L^*H+B\right)}:H\text{ hyperplane in }\mathbb{P}^k\}.
		\end{split}
	\end{equation*}
	
	Now let's consider
	\begin{equation*}
		\begin{split}
			\beta\left(L\setminus L_\red\right)&=\{\overline{\phi\left(F+B\right)}:F\in k\mathfrak{g}_2^1, \ F\text{ is not reduced}\}\\
			&=\{\overline{\phi\left(\phi_L^*H+B\right)}:H\text{ hyperplane in }\mathbb{P}^k\text{ tangent to }\phi_L\left(C \right)\}.
		\end{split}
	\end{equation*}
	
The map $\phi_L\left(C\right)^*\to\beta\left(L\setminus L_\red\right), \ H\mapsto\overline{\phi\left(\phi_L^*H+B\right)}$ is a regular bijection, so $\beta\left(L\right)\subset\overline{\mathcal{G}\left(\mathcal{W}_n\right)}$ contains a copy of the dual hypersurface $\phi_L\left(C \right)^*$, namely $\beta\left(L\setminus L_\red\right)$. 
\end{proof}

Using this theorem with $k=n$ we obtain that $\overline{\mathcal{G}\left(\mathcal{W}_n\right)}$ is a rational variety of dimension $n$.

\begin{cor}\label{tg14}
	Let $C\in\mathscr{M}_g$ be hyperelliptic and $1\leq n\leq g-1$. Let $L$ be the unique $\mathfrak{g}_{2n}^n$ in $C$. Then $\overline{\mathcal{G}\left(\mathcal{W}_n\right)}=\beta\left(L\right)$ is birational to $\mathbb{P}^n$.
\end{cor}

\section{Acknowledgment}

I would like to thank Dr Robert Auffarth for carefully reading
the preliminary version of this paper and several useful comments. I was supported by ANID
Doctorado Nacional 21201016.

\bibliographystyle{alpha}

\begin{thebibliography}{A}
	
	
	
\bibitem{A} A. Andreotti, \textit{On a theorem of Torelli}, Amer. J. of Math. 80 (1958) 801--828.

\bibitem{ACGH} E. Arbarello, M. Cornalba, P. Griffiths and J. Harris, \textit{Geometry of Algebraic Curves}, I, Grundl. der Math. Wiss.,
267, Springer-Verlag, New York, 1985.

\bibitem{BL} C. Birkenhake and H. Lange, \textit{Complex abelian varieties}, Grundl. der Math. Wiss., Springer, 302, second edition, 2010.

\bibitem{BO} A. Borel, \textit{Linear Algebraic Groups}, second edition, Graduate Texts in Math. 126, Springer, New York, 1991.

\bibitem{BR} D. Burns and M. Rapoport, \textit{On the Torelli problem for Kahlerian
K3-surfaces}, Ann. Sci. Éc. Norm. Supér. (4), 8 (1975) 235--274.

\bibitem{CG} J. Carlson and P. Griffiths, \textit{Infinitesimal variations of Hodge structure and the global Torelli problem}, Journées de Géométrie Algébriques d'Angers,
Sijthoff-Noordhoff, (1980) 51-76.

\bibitem{C} F. Catanese, \textit{Infinitesimal Torelli theorems and counterexamples to Torelli problems}, Topics in transcendental algebraic geometry (Princeton, N.J., 1981/1982), Ann. of
Math. Stud., vol. 106, Princeton Univ. Press, Princeton, NJ, 1984, pp. 143--156.

\bibitem{CS} G. Cornell and J. Silverman, \textit{Arithmetic Geometry}, Springer, New York, 1986.

\bibitem{D} A. Dhillon, \textit{A generalized Torelli theorem}, Canad. J. Math. Vol 55, 2 (2003) 248--265.

\bibitem{DO} R. Donagi, \textit{Generic Torelli for projective hypersurfaces}, Compos. Math. 50
(1983) 325--353.

\bibitem{EH} D. Eisebund and J. Harris, \textit{Divisors on general curves and cuspidal rational curves}, Invent. Math. 74 (1983) 371--418.

\bibitem{F} W. M. Faucette, \textit{The Generalized Torelli Problem: Reconstructing a Curve and its Linear Series From its Canonical Map and Theta Geometry}, University of West Georgia (2002).

\bibitem{FS} P. Feller and I. van Santen, \textit{Existence of embeddings of smooth varieties into linear algebraic groups}, J. Alg. Geom. 4 (2022) 729--786.

\bibitem{GH} P. Griffiths and J. Harris, \textit{Principles of Algebraic Geometry}, Wiley Classics
Library, Wiley, 2011.

\bibitem{H} R. Hartshorne, \textit{Algebraic Geometry}, Graduate Text in Mathematics 52, Springer-Verlag, New York, Heidelberg, Berlin, 1977.

\bibitem{RH} R. Hidalgo, \textit{A Kleinian group version of Torelli's Theorem}, Mathematical Sciences, Vol. 21 (2011) 107--115.

\bibitem{K} G. Kempf, \textit{On the geometry of a theorem of Riemann}, Ann. of Math. (2) 98 (1973) 178--185.

\bibitem{L} K. Lauter, \textit{Geometric methods for improving the upper bounds on the number of
rational points on algebraic curves over finite fields}, J. Alg. Geom., vol. 10,
no. 1, (2001) 19--36.

\bibitem{PZ}  G. Pearlstein and Z. Zhang, \textit{A generic global Torelli theorem for certain Horikawa surfaces}, Algebr.
Geom. 6 (2) (2019) 132--147.

\bibitem{SH} I. R. Shafarevich, \textit{Basic Algebraic Geometry 1}, Springer-Verlag, second edition, 1994.

\bibitem{STR1} R. Smith and H. Tapia-Recillas, \textit{The Connection Between Linear Series on Curves and Gauss Maps on Subvarieties of their Jacobians}, Proceedings of the Lefschetz Centennial
Conference, Contemporary Math, Amer. Math. Soc. 58 (1986) 225--238.

\bibitem{STR2} R. Smith and H. Tapia-Recillas, \textit{The Gauss map on subvarieties of Jacobians of curves with $\mathfrak{g}_d^2$'s}, Algebraic Geometry and Complex Analysis, Springer, (1989) 169--180.

\bibitem{S} F. Steffen, \textit{A generalized principal ideal theorem with an application to Brill-Noether theory}, Invent. Math. 132(1), (1998) 73--89.

\bibitem{MT} Montserrat Teixidor i Bigas, \textit{Brill-Noether loci}, Preprint, arXiv:2308.10581, 2023.

\bibitem{T} R. Torelli, \textit{Sulle varietà di Jacobi}, Atti Accad. Naz. Lincei Rend. Lincei Mat. Appl., 22 (1913) 98-103.

\end{thebibliography}

\end{document}